\documentclass[twoside,11pt]{article}
\pdfoutput=1
  
\usepackage[preprint]{jmlr2e}

\usepackage{lastpage}

\jmlrheading{23}{2022}{1-\pageref{LastPage}}{10/22; Revised ....}{...}{21-0000}{Alacaoglu, B\"ohm and Malitsky}

\ShortHeadings{Beyond the golden ratio}{Alacaoglu, B\"ohm and Malitsky}
\firstpageno{1}

\usepackage[utf8]{inputenc} 
\usepackage[T1]{fontenc}    
\usepackage{hyperref}       
\usepackage{url}            
\usepackage{booktabs}       
\usepackage{amsfonts}       
\usepackage{nicefrac}       
\usepackage{microtype}      
\usepackage{xcolor}         

\usepackage{wrapfig}

\usepackage{graphicx}
\usepackage{caption,subcaption}
\newcommand{\R}{\mathbb{R}}
\newcommand{\incr}{\gamma}
\usepackage{wrapfig}
\newcommand{\gb}{\mathbf{g}}

\renewcommand{\L}{\Ecal}

\newcommand{\n}[1]{\|#1 \|}

\newcommand{\lr}[1]{\langle #1\rangle}
\newcommand{\Ecal}{\mathcal{E}}

\newcommand{\zb}{\mathbf{z}}
\newcommand{\zerg}{\mathbf{Z}}
\newcommand{\bzb}{\bar{\mathbf{z}}}
\newcommand{\xb}{\mathbf{x}}
\newcommand{\yb}{\mathbf{y}}
\newcommand{\ub}{\mathbf{u}}
\newcommand{\Gb}{\mathbf{G}}
\newcommand{\Mb}{\mathbf{M}}

\usepackage{mathtools}
\usepackage{hyperref}
\hypersetup{
    colorlinks,
    linkcolor={red},
    citecolor={blue},
    urlcolor={blue}
}

\usepackage{natbib}

\DeclareMathOperator{\Gap}{Gap}

\DeclareMathOperator{\tr}{tr}
\DeclareMathOperator{\clconv}{\overline{conv}}

\newtheorem{assumption}{Assumption}

\usepackage[colorinlistoftodos,prependcaption,backgroundcolor=black!5!white,bordercolor=red]{todonotes}

\usepackage{bm}

\usepackage{thmtools}
\usepackage{thm-restate}

\usepackage{algorithm}
\usepackage[noend]{algpseudocode}

\begin{document}

 \title{Beyond the Golden Ratio for Variational Inequality Algorithms}

\author{\name Ahmet Alacaoglu$^1$ \email alacaoglu@wisc.edu \\
       \addr Wisconsin Institute for Discovery\\
       University of Wisconsin--Madison\\
       Madison, WI, USA
       \AND
       \name Axel B\"ohm$^1$  \email axel.boehm@univie.ac.at \\
       \addr Faculty of Mathematics\\
       University of Vienna\\
       Vienna, Austria
       \AND
       \name Yura Malitsky$^1$ \email y.malitsky@gmail.com\\
       \addr Department of Mathematics\\
       Link\"oping University\\
       Link\"oping, Sweden
}

\maketitle
\begin{abstract}
  We improve the understanding of the \emph{golden ratio algorithm}, which solves monotone variational inequalities (VI) and convex-concave min-max problems via the distinctive feature of adapting the step sizes to the local Lipschitz constants.
  Adaptive step sizes not only eliminate the need to pick hyperparameters, but they also remove the necessity of global Lipschitz continuity and can increase from one iteration to the next. \\[0.15cm]
  We first establish the equivalence of this algorithm with popular VI methods such as reflected gradient, Popov or optimistic gradient descent-ascent in the unconstrained case with constant step sizes. We then move on to the constrained setting and introduce a new analysis that allows to use larger step sizes, to complete the bridge between the golden ratio algorithm and the existing algorithms in the literature. Doing so, we actually eliminate the link between the golden ratio {$\frac{1+\sqrt{5}}{2}$} and the algorithm.
  Moreover, we improve the adaptive version of the
  algorithm, first by removing the maximum step size hyperparameter (an artifact from the analysis) to improve the complexity bound, and second by adjusting it to nonmonotone problems with weak Minty solutions, with superior empirical performance.
\end{abstract}


\footnotetext[1]{Authors are ordered alphabetically.}

\section{Introduction}
With the increasing focus on min-max problems in applications, variational
inequality (VI) algorithms have gained significant attention in machine learning. These algorithms focus on the following VI problem
\begin{equation}\label{eq: vi_prob}
\text{find~} \zb^* \in C \text{~such that~} \langle F(\zb^*), \zb - \zb^*  \rangle \geq 0, \text{~for all~} \zb\in C,
\end{equation}
for a monotone and Lipschitz operator $F\colon C\to \mathbb{R}^d$ and a convex closed set $C \subseteq \mathbb{R}^d$.
The link between this template and convex-concave min-max problems such as
\begin{equation*}
\min_{\xb\in{X}}\max_{\yb\in{Y}} f(\xb, \yb)
\end{equation*}
is well-known and follows by setting
\begin{equation}\label{eq:VI-from-min-max}
  \zb = \binom{\xb}{\yb}, ~~~ F(\zb) = \binom{\nabla_{\xb} f(\xb, \yb)}{-\nabla_{\yb} f(\xb, \yb)}, ~~ C = X \times Y.
\end{equation}
Many algorithms, dating back to 1970s, exist for solving~\eqref{eq: vi_prob}, such as the extragradient method~\citep{korpelevich1976extragradient}, Popov's algorithm~\citep{popov1980modification}, forward-backward-forward~\citep{tseng2000modified}, reflected gradient~\citep{malitsky2015projected,malitsky2018forward}, optimistic gradient descent-ascent (OGDA)~\citep{daskalakis2018training}, dual extrapolation~\citep{nesterov2007dual}.
The algorithm we focus in this paper is the golden ratio algorithm (GRAAL) due to~\citep{malitsky2020golden}, which iterates for $k\geq 0$ as
\begin{equation}\tag{aGRAAL}
  \label{eq:agraal2}
  \begin{aligned}
    \bzb^k &= \frac{\phi-1}{\phi} \zb^k + \frac{1}{\phi} \bzb^{k-1} \\
     \zb^{k+1} &= P_C(\bzb^k - \alpha_k F(\zb^k)).
     \end{aligned}
\end{equation}
with parameter $\phi > 1$ to be defined and a step size sequence $\alpha_k$.
The nonadaptive version of this algorithm uses $\alpha_k = \frac{\phi}{2L}$ where $L$ is the global Lipschitz constant of $F$~\cite[Theorem 1]{malitsky2020golden}.

However, what really separates~\eqref{eq:agraal2} from all the other VI algorithms listed above is the ability to provably use nonmonotone step sizes adapting to local Lipschitzness of $F$, with the adaptive rule
\begin{equation}\label{eq: ss_def}
\alpha_k = \min\left( \incr \alpha_{k-1}, \frac{\phi^2}{4\alpha_{k-2}} \frac{\| \zb^k - \zb^{k-1} \|^2}{\| F(\zb^k) - F(\zb^{k-1})\|^2}, \bar \alpha \right),
\end{equation}
\begin{wrapfigure}{r}{0.40\textwidth}
  \begin{center}
    \includegraphics[width=\linewidth]{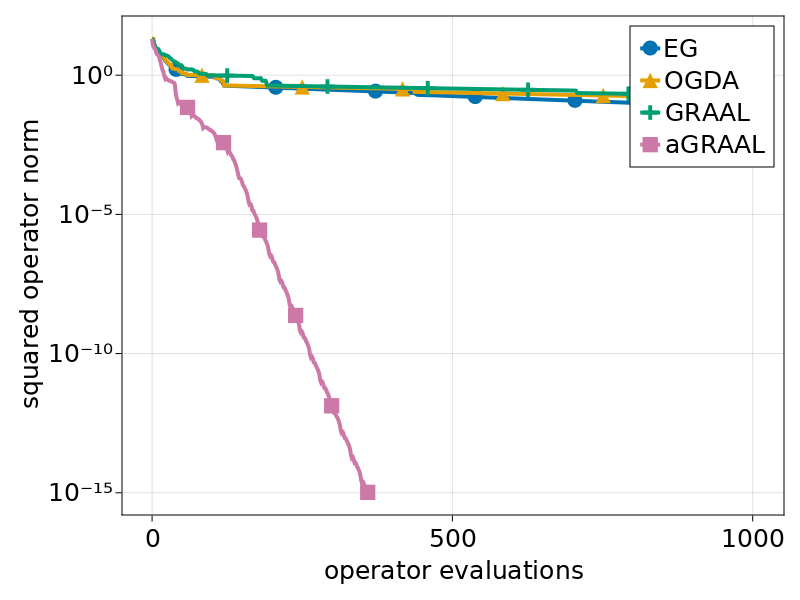}
  \end{center}
  \caption{\small{Policeman\&Burglar matrix game example from~\citep{nemirovski2013mini}.\label{fig:spotlight}}}
\end{wrapfigure}
with $\incr \le \frac{1}{\phi}+ \frac{1}{\phi^2} \in [1,2)$ and $\phi \in (1, \frac{1+\sqrt{5}}{2}]$.
Higher $\phi$ not only gives a larger maximum step size as per~\eqref{eq: ss_def}, but also makes $\bzb^k$ closer to the most recent iterate $\zb^k$ instead of $\bzb^{k-1}$.
The first argument in~\eqref{eq: ss_def} allows increasing step sizes between
iterations and the second estimates a local Lipschitz constant of $F$. The last argument $\bar\alpha$ is required for theoretical reasons in~\citep{malitsky2020golden} and generally picked as a large value in practice so that it will not be effective (see also Section~\ref{sub:adap-graal-no-hyper}).

Compared to the constant step sizes such as $\frac{1}{L}$, the adaptive step sizes make GRAAL highly competitive in practice, in addition to relaxing the assumption of global Lipschitzness of $F$.
Along with its unique empirical performance and generality, GRAAL keeps the good theoretical properties of nonadaptive algorithms: convergence of the sequence to a solution and $O(1/k)$ rate with monotonicity.
A representative plot for the benefit of using step sizes as~\eqref{eq: ss_def} is in Fig.~\ref{fig:spotlight}.

\subsection{Motivation and contributions.}
This paper starts the formal study on understanding the peculiarity of GRAAL among the large sea of VI algorithms.
Towards this goal, we contribute the following results for situating this algorithm in the literature and enhancing its theoretical understanding and practical merit:
\begin{enumerate}
\item In the unconstrained setting with a constant step size, GRAAL is equivalent to Popov's algorithm, OGDA and reflected gradient, when $\phi = 2$. This shows the convergence of GRAAL with $\phi = 2$ for free.
\item With constraints and a constant step size, the established connection is not sufficient and we introduce a novel analysis to prove the convergence of GRAAL with $\phi=2$.
\item We improve the complexity of adaptive GRAAL {from a quadratic dependence on the Lipschitz constant to a linear one,} by eliminating the hyperparameter $\bar \alpha$.
\item We show that adaptive GRAAL has convergence guarantees for nonmonotone problems with weak Minty solutions and show that it reliably converges for some hard instances, even when no other method does so.
\end{enumerate}
\subsection{Notation and preliminaries.}
We say that an operator $F$ is monotone if $\langle F(\xb) - F(\yb), \xb -\yb \rangle \geq 0$ for any $\xb, \yb$ and Lipschitz if $\|F(\xb)- F(\yb) \| \leq L \| \xb-\yb\|$. We denote the projection as $P_{C}(\zb) = \arg\min_{\ub\in C} \| \ub - \zb\|^2$. We denote the golden ratio as $\varphi = \frac{1+\sqrt{5}}{2}$. The (restricted) \emph{dual gap} function (see~\citep{facchinei2003finite,nesterov2007dual}) is a standard notion of suboptimality for VIs and is defined as
\begin{equation}\label{eq: gap_func}
\Gap(\bzb) = \max_{\zb\in S} \langle F(\zb), \bzb - \zb \rangle,
\end{equation}
where $S$ is a compact set.
It is shown in~\citep[Lemma 1]{nesterov2007dual} that this is a valid suboptimality measure, since we will prove that the iterates of the algorithm remain in a compact set.

\begin{assumption}\label{as: as1}
(i) The operator $F\colon C\to \mathbb{R}^d$ is monotone, (ii) $F$ is $L$-Lipschitz, (iii) the set $C$ is convex and closed, (iv)
the set of solutions to~\eqref{eq: vi_prob} is nonempty. 
\end{assumption}

\section{Connection of GRAAL and other VI algorithms}\label{sec: connections}

In this section, we assume that the operator $F$ is monotone and $L$-Lipschitz.
It is well-known that for such operators, resulting for example from min-max problems via~\eqref{eq:VI-from-min-max}, a naive forward evaluation
\begin{equation}\tag{FB}
  \label{eq:forward-backward}
  \zb^{k+1} = P_C(\zb^k - \alpha F(\zb^k))
\end{equation}
\renewcommand*{\theHequation}{notag.\theequation}%
will not converge even for simple bilinear problems with any fixed step size~--- a property which makes even monotone VIs arguably more difficult to solve than the computation of stationary points in nonconvex minimization.
 
Having a nonconvergent scheme~\eqref{eq:forward-backward}, it is tempting to find a simple modification that will ensure convergence. There are two principal approaches to do so: one can try to change either the argument $\zb^k$ or the forward evaluation $F(\zb^k)$ in~\eqref{eq:forward-backward}. Most algorithms opt for the latter option, whereas GRAAL opts for the former as we will see soon. 

A strikingly simple, yet extremely powerful change is to replace $F(\zb^k)$ by $F(\zb^{k+1})$, which leads to
\begin{equation}\tag{PP}
  \label{eq:pp}
  \zb^{k+1} = P_C(\zb^k - \alpha F(\zb^{k+1})),
\end{equation}
\renewcommand*{\theHequation}{notag2.\theequation}%
which is known as the proximal point algorithm~\citep{rockafellar1970convex_analysis}. While this algorithm has great theoretical properties~--- for instance, it converges for any $\alpha>0$~--- it is an implicit method. Computing $z^{k+1}$ is not a simple matter and usually each iteration of~\eqref{eq:pp} requires a call to another subsolver.

One of the earliest, and arguably most popular schemes, is the extragradient
method (EG) which dates back to~\citep{korpelevich1976extragradient}.  It relies on the change of the forward evaluation from $F(\zb^k)$ to $F(P_C(\zb^k-\alpha F(\zb^k)))$. Alternatively we can write the method as
\begin{equation}\tag{EG}
  \label{eq:eg}
    \bzb^k = P_C(\zb^k - \alpha F(\zb^k)); \qquad \zb^{k+1} = P_C(\zb^k - \alpha F(\bzb^k)).
  \end{equation}
\renewcommand*{\theHequation}{notag3.\theequation}%
Interestingly,~\citep{eg-ogda-unified,nemirovski_prox} showed that EG can be viewed as an approximation of~\eqref{eq:pp}. Naturally, however, due to the explicit structure, the method requires an upper bound on the (constant) step size: $\alpha < \frac 1 L$. The point $\bzb^{k}$ is commonly referred to as the \emph{extrapolated point}.

By using an old evaluation of the operator to compute the extrapolated point instead of a new one, we could expect the quality of the iterates not to suffer much, leading to
\begin{equation}\tag{Popov}
  \label{eq:popov}
    \bzb^k = P_C(\zb^k - \alpha F({\bzb^{k-1}})); \qquad  \zb^{k+1} = P_C(\zb^k - \alpha F(\bzb^k)).
  \end{equation}
\renewcommand*{\theHequation}{notag4.\theequation}%
This strategy, proposed by~\citep{popov1980modification}, has similar
guarantees as~\eqref{eq:eg} by requiring a single evaluation of the operator
$F$. This comes at the cost of the need to reduce the step size, leading to the requirement
$\alpha < \frac{1}{2L}$~\citep{hsieh2019convergence}.

About two decades after~\citet{popov1980modification},~\citet{tseng2000modified} proposed to modify EG to only require a single projection:
\begin{equation}\tag{FBF}
  \label{eq:fbf}
    \bzb^k = P_C(\zb^k - \alpha F(\zb^k)); \qquad  \zb^{k+1} = {\bzb^k} - \alpha (F(\bzb^k) {- F(\zb^k)}),
  \end{equation}
\renewcommand*{\theHequation}{notag5.\theequation}%
leading to the forward-backward-forward method, with the same requirement for $\alpha$ as EG:\ $\alpha < \frac 1 L$.

As observed in~\citep{tseng-minimax}, applying Popov's idea to~\eqref{eq:fbf} and
replacing $F(\zb^k)$ by $F(\bzb^{k-1})$, the obtained method can be conveniently
written in one line
\begin{equation}\tag{FoRB}
  \label{eq:forb}
\bzb^{k+1} = P_C (\bzb^k - \alpha (2 F(\bzb^k)- F(\bzb^{k-1}))),
\end{equation}
\renewcommand*{\theHequation}{notag6.\theequation}%
where the bound $\alpha < \frac{1}{2L}$ is imposed. This method was in greater generality (nonconstant step sizes) studied in~\citep{malitsky2018forward} under the name forward-reflected-backward (FoRB), but is more widely known in the unconstrained setting as the optimistic-gradient-descent-ascent (OGDA)~\citep{daskalakis2018training}. Again, \eqref{eq:forb} can be seen as \eqref{eq:forward-backward} where $F(\zb^k)$ is changed to $2F(\zb^k)-F(\zb^{k-1})$.

Another method proposed in~\citep{malitsky2015projected} is projected reflected gradient method (PRG), which solves VI by requiring only one evaluation of $F$ and iterates as
\begin{equation}\tag{PRG}
  \label{eq:prg}
\bzb^{k+1} = P_C(\bzb^k - \alpha F(2 \bzb^k - \bzb^{k-1})).
\end{equation}
\renewcommand*{\theHequation}{notag7.\theequation}%
It is easy to see that the~\eqref{eq:prg} method is equivalent to~\eqref{eq:forb} if the operator $F$ is linear.
The analysis of~\eqref{eq:prg} in~\citep{malitsky2015projected} requires $\alpha< \frac{\sqrt{2}-1}{L}$, which according to the above consideration is not tight when $F$ is linear. Regarding the \eqref{eq:forward-backward} interpretation, we only need to change $F(\zb^k)$ to $F(2\zb^k-\zb^{k-1})$.

A slightly lesser known method, relying on a very similar correction step as the one used in~\eqref{eq:forb}, but applied after the projection, was proposed in~\citep{robert-DR} and is given by
\begin{equation}\tag{shadow-DR}
  \label{eq:shadow-dr}
  \bzb^{k+1} = P_C (\bzb^k - \alpha F(\bzb^k)) - \alpha( F(\bzb^k)- F(\bzb^{k-1})).
\end{equation}
\renewcommand*{\theHequation}{notag8.\theequation}%
In the unconstrained case this method is equivalent to~\eqref{eq:forb}. The analysis in~\citep{robert-DR}, however, requires the more restrictive $\alpha<\frac{1}{3L}$ and even provides a counterexample to show that this dependence is tight. It is not clear where this more conservative dependence on the Lipschitz constant comes from.

Last but not least, our main method of interest is the golden ratio algorithm (GRAAL), introduced in~\citep{malitsky2020golden}, given by
\begin{equation}\tag{GRAAL}
  \label{eq:graal}
    \bzb^k = \frac{\phi-1}{\phi} \zb^k + \frac{1}{\phi} \bzb^{k-1}; \qquad  \zb^{k+1} = P_C(\bzb^k - \alpha F(\zb^k)).
  \end{equation}
  \renewcommand*{\theHequation}{notag9.\theequation}%
Initially, the (nonadaptive) analysis in~\citep{malitsky2020golden} requires choosing $\phi\le \varphi = \frac{1+\sqrt{5}}{2}$ (hence the name with golden ratio), together with a bound on the step size $\alpha \le \frac{\phi}{2L}$. In the following we will show that this bound can be relaxed, in which case we can recover some of the other methods mentioned in this section. Evidently, \eqref{eq:graal} can be seen as a modification of \eqref{eq:forward-backward} with the change $\zb^k$ to $\bzb^k$.

\subsection{GRAAL is averaged PRG}%
\label{sub:}

From the first line of~\eqref{eq:graal} we can rewrite
\begin{equation}
  \label{eq:graal-extrapolated}
  \zb^k = \frac{\phi}{\phi-1} \bzb^k - \frac{1}{\phi-1} \bzb^{k-1} = \bzb^k +  \frac{1}{\phi-1} \left( \bzb^k - \bzb^{k-1} \right).
\end{equation}
Hence instead of viewing $\bzb^k$ as a sequence of averaged iterates we can interpret $\zb^k$ as a sequence of extrapolated iterates.
Plugging~\eqref{eq:graal-extrapolated} into the second line of~\eqref{eq:graal} yields the identity claimed in the title, which gives for $\phi=2$ a $\frac{1}{2}$-averaged version of~\eqref{eq:prg}:
\begin{equation}
  \label{eq:graal-is-avgd-prg}
  \bzb^{k+1}  = \frac{1}{2} \bzb^k + \frac{1}{2} P_C \left(\bzb^k - \alpha F\big(2 \bzb^k - \bzb^{k-1}\big)\right),
\end{equation}
If the problem is unconstrained, GRAAL's $\bzb^k$ sequence corresponds to the one generated by~\eqref{eq:prg}, but with scaled step size.

\subsection{GRAAL is FoRB/OGDA/Popov in the unconstrained case}%
\label{sub:}

While~\eqref{eq:graal} seems fundamentally different from methods such as~\eqref{eq:forb}/OGDA/\eqref{eq:popov} or~\eqref{eq:shadow-dr}, which rely on previous evaluations of the operator $F$, they do turn out to be equivalent in the unconstrained setting, where $\bzb^{k-1} = \zb^k + \alpha F(\zb^{k-1})$, via the simple identity:
\begin{align*}
  \zb^{k+1} &= \bzb^k - \alpha F(\zb^k) \nonumber=\frac{\phi-1}{\phi}\zb^k + \frac{1}{\phi}\bzb^{k-1} - \alpha F(\zb^k)\nonumber
         = \zb^k - \frac{\alpha}{\phi} \Bigl( \phi F(\zb^k) - F(\zb^{k-1}) \Bigr).
\end{align*}
To be precise, the equivalence to OGDA and others holds with $\phi=2$ and for general $\phi$,~\eqref{eq:graal} is equivalent to the generalized version of OGDA, see~\citep{optimism_and_anchoring,axel-weak-minty-ogda}, where $2$ is replaced by $\phi$ and an appropriate scaling of the step size.

\subsection{Summary}\label{sec: wew3}%

In the unconstrained setting, all the methods above collapse to two classes. On one hand, there is~\eqref{eq:eg}/\eqref{eq:fbf}, relying on two gradient evaluations per iteration and a step size constrained by $\frac{1}{L}$. 
On the other, there is the zoo of Popovesque methods:~\eqref{eq:graal},
\eqref{eq:shadow-dr}, \eqref{eq:prg}, \eqref{eq:forb}, \eqref{eq:popov}, and OGDA requiring only one call to $F$ but paying the price of a smaller step size.

The established equivalences give us a proof of convergence for GRAAL with $\phi=2$ in the \emph{unconstrained} case since it reduces to known methods.
However, these relationships are not useful to show the convergence of GRAAL with $\phi=2$ in the \emph{constrained} case which is much more common with min-max problems.
It turns out that standard techniques are not sufficient for such a result, which motivates the next section, dedicated to proving this conclusion. 

\section{GRAAL with $\phi = 2$ for constrained problems}\label{sec: 2raal}

\subsection{Dissection of GRAAL's analysis}\label{sec: dissection}
We sketch the existing analysis of GRAAL in~\citep{malitsky2020golden} and point out to the reason for the restrictive upper bound on $\phi$. Then, we see high level ideas on how to tighten this analysis en route to $\phi=2$.
For convenience, let us define
\begin{equation}\label{eq: we3}
G(\zb^k, \zb) = 2\alpha\langle F(\zb), \zb^k - \zb \rangle ~~~\text{and}~~~
  \Ecal(\zb^{k+1}, \zb) =\frac{\phi}{\phi-1}\Vert \bzb^{k+1}-\zb \Vert^2 + \frac{\phi}{2} \Vert \zb^{k+1}-\zb^k \Vert^2.
\end{equation}
The former is important for showing the rate on the gap function since taking the maximum gives~\eqref{eq: gap_func} for which we will show the $O(1/k)$ rate.
The latter will serve as a Lyapunov (or energy) function.

The analysis of~\citep{malitsky2020golden}, given in Lemma~\ref{lem: one_it} in the Appendix for convenience, results in the following key inequality for $k\geq 1$:
\begin{equation}\label{eq: fe2}
G(\zb^k, \zb) + \Ecal(\zb^{k+1}, \zb) \leq \Ecal(\zb^{k}, \zb)+
 \left(\phi - 1 - \frac{1}{\phi}\right) \| \zb^{k+1} - \bzb^k\|^2
    - \frac{1}{\phi} \| \zb^k - \bzb^{k-1}\|^2.
\end{equation}
Golden ratio appears to make $\phi - 1 - \frac{1}{\phi} = 0$. The analysis in~\citep{malitsky2020golden} discards the \emph{good term} $- \phi \| \zb^k - \bzb^k\|^2$ in the right-hand side to use a telescoping argument, common to methods described in Section~\ref{sec: connections}.

We see that this good term is one index away from the main error term $\| \zb^{k+1} - \bzb^k\|^2$.
However, it is not one index forward, but instead backward, hence it is not immediate how to use it to relax the requirement of $\phi$.
This intuition can be formalized by summing the inequality and using the definition of $\mathcal{E}(\zb^{k+1}, \zb)$ which results in
\begin{align*}
\sum_{i=1}^k G(\zb^i, \zb) &\leq \mathcal{E}(\zb^1, \zb) -\frac{\phi}{2} \| \zb^{k+1} - \zb^k\|^2 + \sum_{i=1}^k \left[\left( \phi - 1 - \frac{1}{\phi} \right) \| \zb^{i+1} - \bzb^i\|^2 - \frac{1}{\phi} \| \zb^i - \bzb^{i-1} \|^2 \right].
\end{align*}
We focus here on the main error term

\begin{equation}\label{eq: fe3}
    -\frac{\phi}{2} \| \zb^{k+1} - \zb^k\|^2 + \sum_{i=1}^k \left[\left( \phi - 1 - \frac{1}{\phi} \right) \| \zb^{i+1} - \bzb^i\|^2 - \frac{1}{\phi} \| \zb^i - \bzb^{i-1} \|^2 \right].
\end{equation}
If this unwieldy expression is bounded, the convergence rate of the gap follows immediately.
It is easy to see that if $\phi = 2$, then~\eqref{eq: fe3} reduces to
\begin{equation}\label{eq: fe4}
- \| \zb^{k+1} - \zb^k\|^2 + \frac{1}{2} \| \zb^{k+1} - \bzb^k\|^2 - \frac{1}{2} \| \zb^1 - \bzb^0\|^2,
\end{equation}
which is not necessarily bounded due to the second term.
In fact, a naive approach can be used for values slightly larger than the golden ratio.
We know by Young's inequality that for any $\tau>0 $
\begin{equation*}
-\| \zb^{k+1} - \bzb^k\|^2 \geq -(1+\tau) \| \zb^{k+1} - \zb^k \|^2 - (1+1/\tau) \| \bzb^k - \zb^k\|^2.
\end{equation*}
It is tedious but straightforward to properly adjust $\phi$ and $\tau$ so that the error term involving $\| \zb^{k+1} - \bzb^k\|^2$ is cancelled by using the negative terms in~\eqref{eq: fe4}. However, this approach is definitely not tight due to spurious use of Young's inequality and only gives $\phi$ values up to $1.77$, whereas we expect $\phi=2$ to be tight, due to the connection with existing methods such as OGDA.

Another obvious remedy  would be to simply \emph{assume} that the term $\| \zb^{k+1} - \bzb^k\|^2$ is bounded, for example, by assuming the sequence $(\zb^k)$ is bounded. 
However, this is not realistic since boundedness of $C$ is a strong assumption, not holding for many min-max problems in practice.
A prevalent example is constrained optimization where neither the primal nor the dual domain is bounded.

In the next section, we \emph{prove} that the sequence $(\zb^k)$ is bounded with $\phi = 2$, which will help us get convergence and rate results.
Instead of the naive approach described above which tries to cancel the error in~\eqref{eq: fe4} by other terms and \emph{loose} inequalities such as Young's, we will analyze the boundedness of the sequence by a novel induction argument on the \emph{tight} inequality in~\eqref{eq: fe3}.

\subsection{GRAAL beyond the golden ratio}
A standard way to prove boundedness of iterates is to identify a Lyapunov function (nonnegative and nonincreasing) including terms such as $\|\zb^k - \zb\|^2$.
An example is $\Ecal$ in~\eqref{eq: we3}.
While this can be done with $\phi \leq \varphi$, it is unclear if it is possible with $\phi = 2$, due to the issue described in Section~\ref{sec: dissection}.
To go around this difficulty, we have to use nonstandard tools and forgo the Lyapunov function argument and analyze boundedness of $(\zb^k)$ directly via induction on the key inequality~\eqref{eq: fe2}.

\begin{theorem}\label{thm:bounded}
Let Assumption~\ref{as: as1} hold and let $\phi = 2$, $\alpha \leq  \frac{1}{L}$ in~\ref{eq:graal}. Then, we have that $(\zb^k)$ and $(\bzb^k)$ are bounded sequences. In particular, we have
\begin{equation*}
\| \bzb^k - \zb^* \|^2 \leq 4\left(\| \zb^1 - \zb^* \|^2 + \| \zb^0 - \zb^* \|^2\right) \leq 12 \| \zb^0 - \zb^*\|^2.
\end{equation*}
\end{theorem}

While we provide a formal proof for the case $\phi=2$ here for simplicity, we also supply a computer aided proof via semi-definite programming, see Appendix~\ref{sec:appendix}, which in addition covers the case for $\phi<2$ and provides a better constant for $\phi=2$.

\begin{proof}
In~\eqref{eq: fe2}, we set $\zb=\zb^*$ and use $G(\zb^k, \zb^*)\geq 0$ by~\eqref{eq: vi_prob}.
Next, we sum~\eqref{eq: fe2} with $\phi=2$ which gives us
\begin{align}
2\n{\bzb^{k+1}-\zb^*}^2 + \n{\zb^{k+1}-\zb^k}^2
&\leq 2\n{\bzb^{1}-\zb^*}^2  + \frac12 \n{\zb^1-\zb^0}^2+  \frac 1 2\n{\zb^{k+1}-\bzb^k}^2 \notag \\
&= \| \zb^1 - \zb^*\|^2+\| \zb^0 - \zb^*\|^2  +\frac 1 2\n{\zb^{k+1}-\bzb^k}^2,\label{beauty_main}
\end{align}
where we used $\zb^0 = \bzb^0$. Define $R^2=\| \zb^1 - \zb^*\|^2 + \|\zb^0 -\zb^*\|^2$.  For the inductive step, assume that $\n{\bzb^i-\zb^*}\leq 2R$ for all $i\leq k$.

We will transform the error term $\frac{1}{2}\| \zb^{k+1} - \bzb^k\|^2$ in~\eqref{beauty_main} so that the other terms in the left-hand side of~\eqref{beauty_main} can be used for (partial) cancellation. 
By the paralellogram law and definition $\bzb^k=\frac{1}{2}\zb^k + \frac{1}{2} \bzb^{k-1}$, we have
\begin{align*}
\frac{1}{2}\| \zb^{k+1} - \bzb^k\|^2 &= \| \zb^{k+1} - \zb^k \|^2 + \| \zb^k - \bzb^k\|^2 - \frac{1}{2}\| \zb^{k+1} - 2\zb^k + \bzb^k \|^2 \\
&= \| \zb^{k+1} - \zb^k \|^2 + \|\zb^k - \bzb^k\|^2 - 2\| \bzb^{k+1} - \zb^k \|^2.
\end{align*}
Plugging in this equality to~\eqref{beauty_main} and using the definition of $\bzb^k$ gives us
\begin{align}\label{beauty2_main}
2\n{\bzb^{k+1}-\zb^*}^2 + 2\n{\bzb^{k+1}-\zb^k}^2
&\leq R^2  +  \n{\zb^k-\bzb^{k}}^2 = R^2 + \frac 14\n{\zb^k-\bzb^{k-1}}^2.
\end{align}
The left-hand side is in a suitable form to apply another paralellogram law and obtain
\begin{equation*}
2\n{\bzb^{k+1}-\zb^*}^2 +2\n{\bzb^{k+1}-\zb^k}^2 = \n{\zb^k-\zb^*}^2 + 4\left\|\bzb^{k+1} - \frac{\zb^*+\zb^k}{2}\right\|^2.
\end{equation*}
After combining this equality with \eqref{beauty2_main}, it follows that
\begin{align}\label{better_main}
  \n{\zb^k-\zb^*}^2 + 4\left\|\bzb^{k+1} - \frac{\zb^*+\zb^k}{2}\right\|^2 \leq  R^2 + \frac 1 4\n{\zb^{k}-\bzb^{k-1}}^2.
\end{align}
We are now at the most critical point of the proof.
For induction, we will combine the second term in the right-hand side of~\eqref{better_main} with the terms in the left-hand side to obtain $\| \bzb^{k-1} - \zb^*\|^2$.
Now continue with the following identity which can be verified by simple expansion
\begin{align*}
  \frac 14\n{\zb^k-\bzb^{k-1}}^2 - \n{\zb^k-\zb^*}^2 &= -\frac 34\left\|\zb^k - \frac{4}{3} \zb^* + \frac{1}{3}\bzb^{k-1} \right\|^2+\frac 13 \n{\bzb^{k-1}-\zb^*}^2.
\end{align*}
Even though it is difficult to see the significance of this identity, a high level intuition is noticing that $\zb\mapsto \frac{1}{4} \| \zb - \bzb^{k-1}\|^2 - \| \zb - \zb^*\|^2$ is maximized at $\zb = \frac{4}{3} \zb^* - \frac{1}{3}\bzb^{k-1}$.
We can then rewrite \eqref{better_main} as
\begin{align}\label{better3_main}
  3 \left\|\frac{\zb^k - \zb^*}{2}  + \frac{\bzb^{k-1}-\zb^*}{6}\right\|^2 + 4\left\|\bzb^{k+1} - \frac{\zb^*+\zb^k}{2}\right\|^2 \leq  R^2 + \frac 1 3\n{\bzb^{k-1}-\zb^*}^2.
\end{align}
Let $a = \zb^k-\zb^*$, $b = \bzb^{k-1}-\zb^*$. Then using this notation in~\eqref{better3_main}, taking square root of both sides, using triangle inequality and the inductive assumption ($\| b\| \leq 2R$) implies that
\begin{equation*}
\left\|\frac 1 2 a + \frac 1 6 b\right\|^2 \leq \frac{R^2}{3} + \frac 1 9 \n{b}^2 \quad \implies \quad \frac 12\| a\| \leq \sqrt{\frac{R^2}{3}+\frac 1 9 \|b\|^2} + \frac 1 6 \| b\|\leq \frac{\sqrt{7} + 1}{3} R.
\end{equation*}
On the other hand, \eqref{better3_main} along with inductive assumption gives us
\begin{equation*}
\left\|\bzb^{k+1}-\frac{\zb^*+\zb^k}{2}\right\|^2 \leq \frac{R^2}{4}  + \frac{1}{12}\n{\bzb^{k-1}-\zb^*}^2 \quad \implies \quad \left\|\bzb^{k+1}-\frac{\zb^*+\zb^k}{2}\right\| \leq \sqrt{\frac{7R^2}{12}}.
\end{equation*}
We can now combine the last two estimations with triangle inequality to complete the induction
\begin{align}
  \n{\bzb^{k+1}-\zb^*}&\leq \left\|\bzb^{k+1}-\frac{\zb^*+\zb^k}{2}\right\| + \left\|\frac{\zb^*+\zb^k}{2}-\zb^*\right\| = \left\|\bzb^{k+1}-\frac{\zb^*+\zb^k}{2}\right\| + \frac 1 2\n{\zb^k-\zb^*} \nonumber \\
                  &\leq \sqrt{\frac{7}{12}}\cdot R + \frac{\sqrt 7 + 1}{3}\cdot R < 2R.
\end{align}
Hence, by induction we proved that $(\bzb^k)$ is bounded. The definition $\zb^k = 2\bzb^k - \bzb^{k-1}$ shows that $(\zb^k)$ is also bounded. The final inequality uses $R^2 \leq 3 \| \zb^0 - \zb^*\|^2$, which is shown in Lemma~\ref{lem: rto4} in Appendix~\ref{sec:appendix}.
\end{proof}

\begin{corollary}\label{cor: as3}
Let Assumption~\ref{as: as1} hold, $\phi = 2$, $\alpha = \frac{1-\varepsilon}{L}$ for any $\varepsilon\in(0, 1)$ in \ref{eq:graal} and $\mathbf{Z}^k = \frac{1}{k}\sum_{i=1}^k\zb^i$. Then $(\zb^k)$ converges to a solution of~\eqref{eq: vi_prob} and
\begin{equation*}
\Gap\left( \mathbf{Z}^k \right) \leq \frac{32L}{(1-\varepsilon)k}\| \zb^0 - \zb^*\|^2.
\end{equation*}
\end{corollary}
\begin{proof}
We begin by proving the first part of the corollary.
By the boundedness of the iterates, proved in Theorem~\ref{thm:bounded}, we have that $(\zb^k)$ has a convergent subsequence, say $(\zb^{k_i})$ and for some $\tilde \zb \in C$, we have that $\zb^{k_i} \to \tilde \zb$. We next show that $\tilde \zb$ is a solution of~\eqref{eq: vi_prob}. For this, first notice that by summing the result of Lemma~\ref{lem: one_it} with $\zb=\zb^*$, $\phi=2$ and using the boundedness of $\|\bzb^{k+1} - \zb^k\|^2$ derived in Theorem~\ref{thm:bounded}, we have that
\begin{equation*}
\sum_{k=0}^\infty \|\zb^k - \zb^{k-1} \|^2 < +\infty.
\end{equation*}
Young's inequality gives
\begin{align*}
\| \zb^k - \bzb^k \|^2 &\leq 3 \| \zb^{k} - \zb^{k-1} \|^2 + \frac{3}{2}\|  \bzb^k - \zb^{k-1} \|^2 \\
&\leq \frac{15}{4} \| \zb^{k} - \zb^{k-1} \|^2 + \frac{3}{4}\|  \bzb^{k-1} - \zb^{k-1} \|^2,
\end{align*}
where the second inequality is by the definition $\bzb^k =\frac{1}{2}\zb^k + \frac{1}{2}\bzb^{k-1}$. Rearranging and summing this inequality gives
\begin{equation*}
\sum_{k=0}^\infty \| \zb^k - \bzb^k \|^2 \leq \sum_{k=0}^\infty 15\| \zb^k - \zb^{k-1} \|^2 < +\infty.
\end{equation*}
where we used $\zb^{-1} = \bzb^{-1}$ for getting the first inequality.

As a result, we have that $\zb^{k+1} - \zb^k \to 0$ and $\bzb^k - \zb^k \to 0$ and hence $\bzb^{k_i} \to \tilde \zb$ and $\zb^{k_i+1}\to\tilde \zb$.
We can take the limit in the prox-inequality
\begin{equation*}
  \langle \zb^{k+1} - \bar{\zb}^k + \alpha F(\zb^k), \zb - \zb^{k+1} \rangle \geq 0
\end{equation*}
and use continuity of $F$ to get that $\tilde \zb$ is a solution. For simplicity, set $\tilde \zb = \zb^*$, where $\zb^*$ is an arbitrary solution.

The result of Lemma~\ref{lem: one_it} also gives the following inequality after using $\phi=2$ and $\zb=\zb^*$:
\begin{equation*}
2 \| \bzb^{k+1} - \zb^*\|^2 + \| \zb^{k+1} - \zb^k \|^2 - \frac{1}{2} \| \zb^{k+1} - \bzb^k \|^2 \leq 2 \| \bzb^{k} - \zb^*\|^2 + \| \zb^{k} - \zb^{k-1} \|^2 - \frac{1}{2} \| \zb^{k} - \bzb^{k-1} \|^2.
\end{equation*}
Since $\|\zb^{k+1} - \bzb^k \|^2$ is bounded, the sequence on the left-hand side of the inequality is lower bounded and nonincreasing, hence convergent. In summary, we have that
\begin{equation*}
\lim_{k\to \infty} 2 \| \bzb^{k+1} - \zb^*\|^2 + \| \zb^{k+1} - \zb^k \|^2 - \frac{1}{2} \| \zb^{k+1} - \bzb^k \|^2 \text{~exists.}
\end{equation*}
Since $\| \zb^{k+1} - \zb^k \|^2 \to 0$ and $\| \zb^{k+1} - \bzb^k \|^2=4\|\bzb^{k+1} - \zb^{k+1} \|^2 \to 0$, we further deduce that
\begin{equation*}
\lim_{k\to \infty} \| \bzb^{k+1} - \zb^*\|^2 \text{~exists}.
\end{equation*}
Since we previously showed that $\bzb^{k_i} - \zb^* \to 0$ for an arbitrary $\zb^*$, we conclude that $\bzb^k - \zb^* \to 0$. It is easy to see that $\zb^k - \zb^*\to 0$ as well since we have $\bzb^k - \zb^k \to 0$.

Now we turn our attention to the convergence rate.
We start by collecting some estimations from Theorem~\ref{thm:bounded}.
In particular, the main conclusion of the theorem was that
\begin{equation}\label{eq: dsd4}
\| \bzb^k - \zb^*\|^2 \leq 4R^2 = 12 \| \zb^0 - \zb^*\|^2.
\end{equation}
Moreover, from~\eqref{better_main}, we have that
\begin{align}
  \n{\zb^k-\zb^*}^2 &\leq  R^2 + \frac 1 4\n{\zb^{k}-\bzb^{k-1}}^2\notag \\
&\leq R^2 + \frac 1 2\n{\zb^{k}-\zb^*}^2 + \frac{1}{2} \| \bzb^{k-1} - \zb^*\|^2.\notag
\end{align}
which together with~\eqref{eq: dsd4} implies
\begin{equation}\label{eq: dsd5}
  \n{\zb^k-\zb^*}^2 \leq 6R^2 = 18 \| \zb^0 - \zb^*\|^2.
\end{equation}
Since the guarantee on the gap is on the average of the sequence $(\zb^k)$, we need to set $S$ such that it will contain $(\zb^k)$ to use the result of~\citep[Lemma 1]{nesterov2007dual}. Let
\begin{equation*}
S = \{ \zb \in C\colon \| \zb - \zb^*\|^2\leq 18 \| \zb^0 - \zb^*\|^2 \}.
\end{equation*}
We set $\phi=2$ on Lemma~\ref{lem: one_it}, sum this inequality, divide by $k$ and take maximum over $S$ to obtain
\begin{equation*}
\Gap(\mathbf{Z}^k) \leq \frac{1}{2\alpha}\left( \max_{\zb\in S} 2\| \bzb^1 - \zb \|^2 + \frac{1}{2}\| \zb^1 - \zb^0\|^2 + \frac{1}{2}\|\zb^{k+1} - \bzb^{k} \|^2\right),
\end{equation*}
where the $\frac{1}{2\alpha}$ factor on the right-hand side is since $G(\zb^k, \zb) = 2\alpha \langle F(\zb), \zb^k - \zb \rangle$ (see~\eqref{eq: we3}). 

By using $\bzb^1 = \frac{1}{2} \zb^1 + \frac{1}{2} \zb^0$ due to $\bzb^0 = \zb^0$, we get
\begin{equation}\label{eq: hew2}
\Gap(\mathbf{Z}^k) \leq \frac{1}{2\alpha}\left( \max_{\zb\in S} (\| \zb^1 - \zb \|^2+\| \zb^0 - \zb \|^2) + \frac{1}{2}\|\zb^{k+1} - \bzb^{k} \|^2\right).
\end{equation}
We will now use~\eqref{eq: dsd4} and~\eqref{eq: dsd5} to get
\begin{equation}\label{eq: hew3}
\frac{1}{2}\|\zb^{k+1} - \bzb^{k} \|^2 \leq \| \zb^{k+1} - \zb^*\|^2+\| \bzb^{k} - \zb^*\|^2 \leq 30 \| \zb^0 - \zb^*\|^2.
\end{equation}
Moreover, we have for any $c$ that
\begin{align}
\max_{\zb\in S} (\| \zb^1 - \zb \|^2+\| \zb^0 - \zb \|^2) &\leq \max_{\zb\in S} 2(1+c) \| \zb - \zb^*\|^2 + (1+1/c) (\| \zb^1 - \zb^*\|^2 + \|\zb^0 - \zb^*\|^2) \notag \\
&\leq (36(1+c)+ 3(1+1/c)) \| \zb^0 - \zb^*\|^2.\label{eq: hew4}
\end{align}
We apply~\eqref{eq: hew3} and~\eqref{eq: hew4} in~\eqref{eq: hew2} and pick $c$ to minimize the corresponding term, to obtain the result.
\end{proof}

Increasing $\phi$ from the golden ratio to $2$ for the fixed step size regime not
only emphasizes an interesting connection to OGDA, but also consistently
improves empirical performance for monotone problems, see
Fig.~\ref{fig:constant}.  Even though the adaptive version of GRAAL is typically
superior in practice as per the experiments of~\cite{malitsky2020golden}, we
want to point out that with constant step sizes, the need to pick $\phi \le \varphi < 2$
caused GRAAL to perform worse than methods such as OGDA.  The main merit of our
result is showing that this is not the case.

\begin{figure}[ht]
  \begin{subfigure}[c]{.3\textwidth}
    \includegraphics[width=\linewidth]{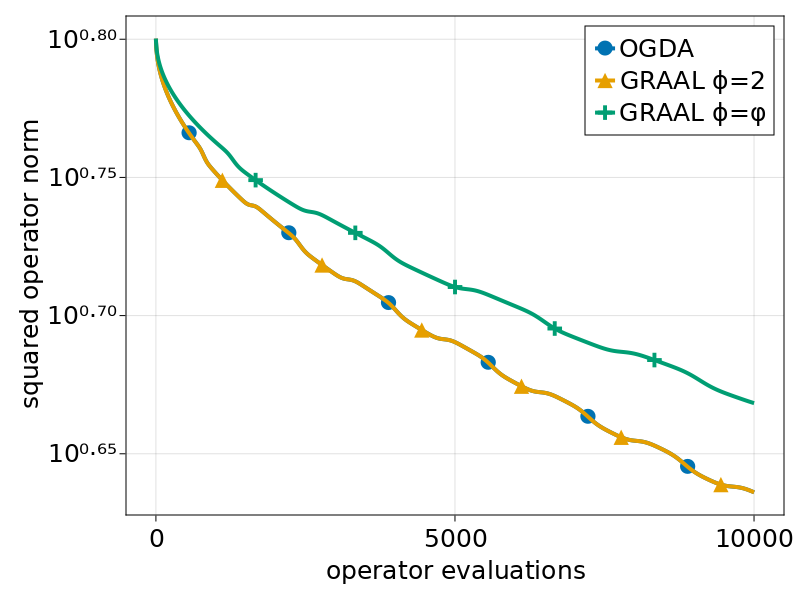}
  \end{subfigure}
  \hspace{.3cm}
  \begin{subfigure}[c]{.3\textwidth}
    \includegraphics[width=\linewidth]{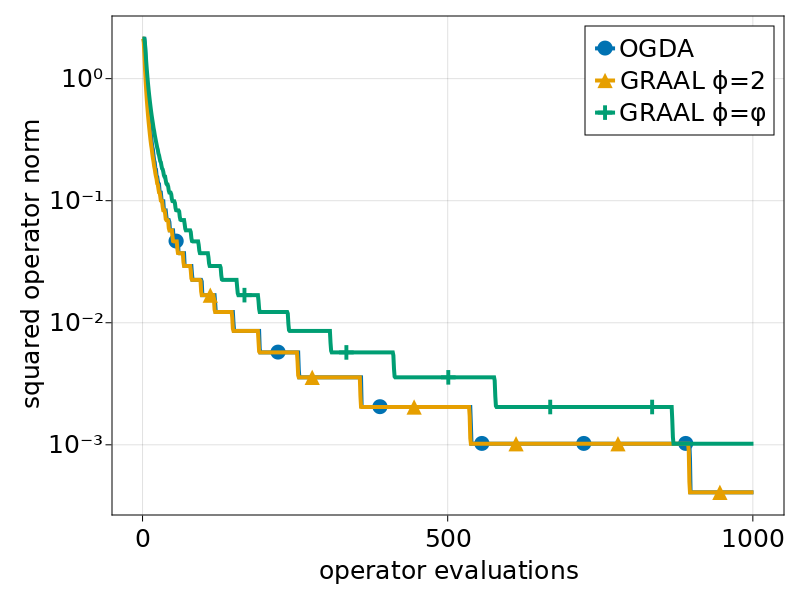}
  \end{subfigure}
  \hspace{.3cm}
  \begin{subfigure}[c]{.3\textwidth}
    \includegraphics[width=\linewidth]{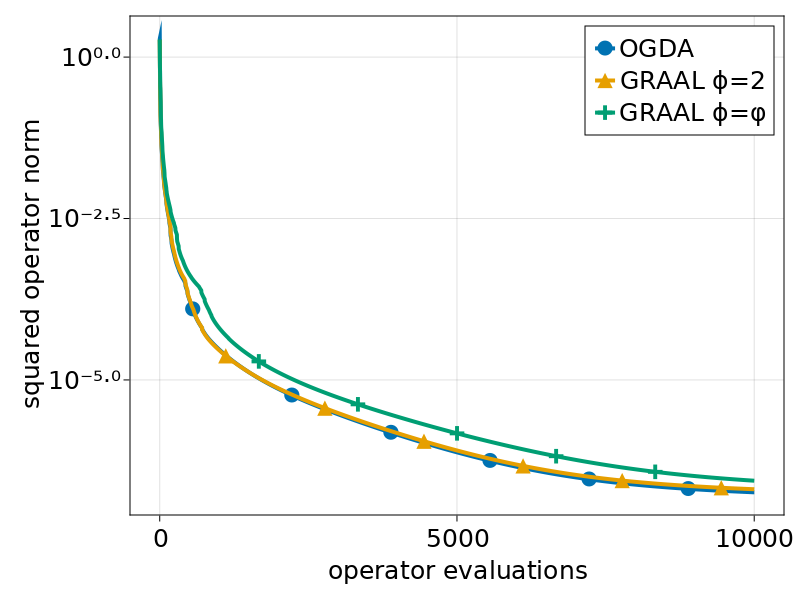}
  \end{subfigure}
  \centering
  \caption{{\small Left: The linearly constrained QP in~\citep{acc-gradient-norm-minimax}. Middle: Test matrix given in~\citep{nemirovski-stochastic-subgradient}. Right: Randomly generated matrix game. Interestingly, even with constraints, GRAAL and OGDA perform almost identically.\label{fig:constant}}}
\end{figure}

\section{Adaptive GRAAL: Removing hyperparameters and improving complexity}
\label{sub:adap-graal-no-hyper}
We first recall aGRAAL, proposed in~\citep{malitsky2020golden}:
\begin{equation}\tag{aGRAAL}
  \label{eq:agraal}
  \begin{aligned}
    \bzb^k &= \frac{\phi-1}{\phi} \zb^k + \frac{1}{\phi} \bzb^{k-1} \\
    \zb^{k+1} &= P_C(\bzb^k - \alpha_k F(\zb^k)),
  \end{aligned}
\end{equation}
\renewcommand*{\theHequation}{notag10.\theequation}%
where $\alpha_k$ is picked as in~\eqref{eq: ss_def} and $\phi \in \left(1, \frac{1+\sqrt{5}}{2} \right)$. As mentioned in~\citep{malitsky2020golden}, the hyperparameter $\bar\alpha$ in~\eqref{eq: ss_def} is only for theoretical purposes and the suggested choice in practice is taking $\bar \alpha$ very large so that it will be ineffective in~\eqref{eq: ss_def}.
However, there are two sides to this coin from a complexity point of view as we see now. Recall that~\citep{malitsky2020golden} proved that the iterates remain in a bounded region and established the following worst case rate for the ergodic iterates
\begin{equation}\label{eq: gr4}
\Gap\left(\zerg^k \right) \leq \frac{4L^2 \bar\alpha}{k\phi^2} \underbrace{\max_{\zb \in S} \left( \frac{\phi}{\phi-1} \| \zb^1 -\zb\|^2 + \frac{\theta_0}{2} \| \zb^0 - \zb\|^2 \right)}_{=: D},
\end{equation}
where $L$ denotes the (unknown) Lipschitz constant of $F$ over this bounded region.
On one hand, taking $\bar\alpha$ too large could make this rate vacuous.
On the other, taking $\bar\alpha$ small may prevent taking large step sizes in~\eqref{eq: ss_def}.
Another aspect of this bound is that the dependence on $L$ is suboptimal, since most VI methods including \emph{nonadaptive} GRAAL result in the rate $O(L/k)$.
Obtaining linear dependence on $L$ suggests taking $\bar\alpha$ as a large multiple of $\frac{1}{L}$, which not only requires the knowledge of $L$, but also could make the constant of the rate large as per the discussion above.
This suboptimal worst-case complexity result may be seen as the cost of adaptivity.
To avoid this conflict regarding the choice of $\bar\alpha$ in practice and the resulting complexity, we propose to remove $\bar\alpha$ in \eqref{eq: ss_def} and provide an analysis with the simpler step size rule
\begin{equation}\label{eq: alpha_new}
\alpha_k = \min\left( \gamma \alpha_{k-1}, \frac{\phi^2}{4 \alpha_{k-2}} \frac{\| \zb^k - \zb^{k-1}\|^2}{\| F(\zb^k) - F(\zb^{k-1})\|^2} \right),
\end{equation}
with $\gamma = \frac{1}{\phi} + \frac{1}{\phi^2}$ to complement the empirical success of the algorithm in the experiments of~\citep{malitsky2020golden} with the choice~\eqref{eq: alpha_new}.
This rule not only removes the hyperparameter $\bar\alpha$ but as we show in the
next theorem, it also results in the rate $O(L/k)$ which is only a small
constant times worse than the rate of the nonadaptive method (see
Remark~\ref{rem: rem_constants} for a precise statement). 
As a result, this is a much smaller cost to pay for the worst-case complexity
with adaptivity. With the proposed change, we obtain~Alg.~\ref{alg:agraal}.

We will start with the one iteration analysis from~\citep{malitsky2020golden} which does not need any modification, so we state the result with a very brief proof to make the connection easy to follow. Note that the spurious $\bar\alpha$ term is not used in the analysis of~\cite[Theorem 2]{malitsky2020golden} in this result.

 \begin{algorithm}[H]
   \caption{aGRAAL}\label{alg:agraal}
   \algorithmicrequire{ $\bzb^{0} = \zb^0 \in C$, $\phi \in \left(1, \frac{1+\sqrt{5}}{2} \right)$, $\gamma \in \left(1, \frac{1}{\phi} + \frac{1}{\phi^2}\right]$, $\alpha_0 >0 $, $\theta_0 = \phi$.}
   \begin{algorithmic}[1]
      \State $\zb^1  = P_C(\zb^0 - \alpha_0 F(\zb^0))$ \hfill \textcolor{gray}{possibly by linesearch}
     \For{$k\geq1$}
     \State $\alpha_k =\min\left( \gamma \alpha_{k-1}, \frac{\phi \theta_{k-1}}{4\alpha_{k-1}}\frac{\| \zb^k - \zb^{k-1}\|^2}{\|F(\zb^k) - F(\zb^{k-1})\|^2}  \right)$ \label{step: step_adap}
        \State  $\bzb^k = \frac{\phi-1}{\phi} \zb^k + \frac{1}{\phi}\bzb^{k-1}$
        \State $\zb^{k+1} = P_C(\bzb^k - \alpha_k F(\zb^k))$
        \State $\theta_{k} = \frac{\alpha_k}{\alpha_{k-1}}\phi$
     \EndFor{}
   \end{algorithmic}
 \end{algorithm}

\begin{lemma}{(Consequence of~\cite[Theorem 2]{malitsky2020golden})}\label{lem: one_it_adap}
  Let Assumption~\ref{as: as1}(i, iii, iv) hold and $F$ be locally Lipschitz. Let $\phi \in \left(1, \frac{1+\sqrt{5}}{2}\right)$, $\gamma \le \frac{1}{\phi}+\frac{1}{\phi^2}$ and $\zerg^k = \frac{1}{\sum_{i=1}^k\alpha_i}\sum_{i=1}^k \alpha_i\zb^i$. Then, the iterates $(\zb^k)$ of Alg.~\ref{alg:agraal} are bounded and we have
  \begin{equation}\label{eq: yte4}
    \Gap(\zerg^k) \leq \frac{D}{2\sum_{i=1}^k \alpha_i} \quad \forall k\ge 1.
  \end{equation}
\end{lemma}
\begin{proof}
  In~\cite[Theorem 2]{malitsky2020golden} inequality (35) is only valid for $k\ge 2$. However due to our definition of $\zb^1$ and since $\zb^0 = \bzb^0$, it already holds for $k\geq 1$. Then, we can unroll the recursion another step until iteration $k=1$ instead of $k=2$.
\end{proof}

\subsection{Lower bounding the sum of step sizes}%
Looking at the bound in~\eqref{eq: yte4}, we notice that in order to derive a
complexity result, we require a lower bound for the sum of the step sizes.  In
the original proof in~\citep{malitsky2020golden} such a bound is ensured by
enforcing each $\alpha_i$ to be lower bounded by using $\bar \alpha$ to derive a $O(1/k)$ rate.

\medskip

For aGRAAL the consecutive step sizes also depend on
the initial $\alpha_0$. Since the method is entirely adaptive and we do not
assume any a priori knowledge about $F$, we want to make sure that $\alpha_0$ is
not too small and not too large. To this end, for the initialization, we
recommend to use the linesearch procedure described below.

For brevity, we will  write
$L_k = \frac{\n{\zb^k - \zb^{k-1}}}{\n{F(\zb^k) - F(\zb^{k-1})}}$. Also, let $L$
be the Lipschitz constant of $F$ over the \emph{bounded} set
$\clconv\{\zb^0, \zb^1,\dots\}$.  It trivially holds that $L\geq L_k$ for all $k$.
Without loss of generality, we assume that $L\geq 1$.

Let us choose $\alpha_0$ via linesearch as follows: set $\alpha_0$ to the largest number in $\{\gamma^{-i}\colon i=0,1,\dots\}$ (note here the use of $\gamma$ given in Alg.~\ref{alg:agraal}) such that
\begin{equation}\label{eq:lns}
\begin{aligned}
  &\zb^1 = P_C(\zb^0 - \alpha_0 F(\zb^0))\\
  &\alpha_0 \n{F(\zb^1) - F(\zb^0)}\leq \frac{\phi}{2} \n{\zb^1 - \zb^0}.
\end{aligned}
\end{equation}
The second equation gives the upper bound for $\alpha_0$: $\alpha_0\leq \frac{\phi}{2L_1}$.
Note that the linesearch always terminates because $F$ is locally Lipschitz, and we can immediately obtain the following statements, which we will use later in Lemma~\ref{lem: main_sec4}.
We have either $\alpha_0 = \gamma^{-0}=1>\frac{\phi}{2L}\ge \frac{\phi}{2 \gamma L}$ or that
$\gamma \alpha_0=\gamma \cdot \gamma^{-i}$ violates the inequality above, that is
\[\gamma \alpha_0 \n{F(\zb^1) - F(\zb^0)}> \frac{\phi}{2} \n{\zb^1 - \zb^0},\]
which also implies that $\alpha_0 \geq \frac{\phi}{2\gamma L_1} \ge \frac{\phi}{2 \gamma L}$. Moreover, from the
algorithm's update for $\alpha_1$ we have either
$\alpha_1 = \gamma \alpha_0\geq \frac{\phi}{2L}$
or
\[\alpha_1 =\frac{\phi\theta_0}{4\alpha_0 L_1^2} = \frac{\phi^2}{4\alpha_0L_1^2},\]
which implies that
$\alpha_1 + \alpha_0\geq 2\sqrt{\alpha_1\alpha_0}\geq \frac{\phi}{L_1}$.
 Using the upper bound for $\alpha_0$, we deduce that
$\alpha_1\geq\frac{\phi}{2L_1} \geq \alpha_0$.  To conclude, the suggested
choice of $\alpha_0$ in \eqref{eq:lns} ensures that
 \[\alpha_1\geq \frac{\phi}{2L} \quad \text{and}\quad  \alpha_1\geq \alpha_0.\]

The following lemma is our main technical contribution in
Section~\ref{sub:adap-graal-no-hyper}, which will lead to Theorem~\ref{th: sec4}
when combined with the previous lemma.
\begin{lemma}\label{lem: main_sec4}
Let $\alpha_0$ satisfy \eqref{eq:lns} and  $c = \min_{m\geq 2} \frac{\phi}{m}\sqrt{\frac{\gamma^{m-1}-1}{\gamma-1}}$. Then, we have
\begin{equation*}
\sum_{i=1}^k \alpha_i \geq \frac{(k-1)c}{L}.
\end{equation*}
\end{lemma}
  As mentioned in the beginning of this section, we can no longer rely on a lower bound for every individual step size but want to use their structure to lower bound their sum directly.
  Specifically, whenever option two is active for $\alpha_i$ (meaning that the second component in~\eqref{eq: alpha_new} attains the minimum), then
  we can easily show that $\alpha_{i-2}+\alpha_i\geq \frac{\phi}{L}$. On the
  other hand, if option two is not present for a long time, then we will have a
  geometric progression $\alpha_i, \alpha_i\gamma,\alpha_i\gamma^2,\dots$, whose
sum is also easy to bound. These are two main ingredients, however
combining the two requires an intricate technical argument.

\begin{proof}
We call $\alpha_k = \gamma \alpha_{k-1}$ and $\alpha_k = \frac{\phi^2}{4\alpha_{k-2}L_{k}^2}$ as
the first and second options that the step size $\alpha_k$ can take
respectively.

\paragraph{Claim 1.} If $\alpha_k$ satisfies the second option, for $k\geq 2$, then
$\alpha_{k-2} + \alpha_k\geq \frac{\phi}{L_k}\geq \frac{\phi}{L}$. \\
This follows directly from  $a+b\geq 2\sqrt{ab}$.

\paragraph{Claim 2.} $0<c\leq\frac{\phi}{2}$.\\  The first inequality follows from $\gamma > 1$ and
$\lim_{m\to \infty}\frac \phi m \sqrt{\frac{\gamma^{m-1} - 1}{\gamma - 1}} = \infty$. The
second one follows from setting $m=2$ in the definition of
$c =\min_{m\geq 2} \frac{\phi}{m}\sqrt{\frac{\gamma^{m-1}-1}{\gamma-1}} $.

\bigskip 
We wish to show that
\begin{equation*}
\sum_{i=1}^k \alpha_i \geq \frac{c(k-1)}{L}.
\end{equation*}
If there does not exist an $\alpha_i$ smaller than $\frac{c}{L}$, then we are done. So assume that such an $\alpha_i$ exists.

For $s\geq1$, let us call by a \emph{tail} a maximal subsequence of consecutive
elements $\alpha_{s+1},\dots, \alpha_{s+m}$ such that it starts from $\alpha_{s+1}\geq \frac{c}{L}$ and the rest of the elements are all smaller:
\[\alpha_{j}< \frac{c}{L} \quad \text{for all $j=s+2,\dots, s+m$\quad  and\quad
    $\alpha_{s+m+1}\geq \frac cL$}.\]
Notice that for every such a tail, $\alpha_{s+2}<\alpha_{s+1}$, which means that the second option for
$\alpha_{s+2}$ is active. By Claim~1, this implies that $\alpha_s\geq \frac{c}{L}$.  We call this element $\alpha_s$ as a \emph{head}. Thus, for every tail
$\alpha_{s+1},\dots,\alpha_{s+m}$ there is a preceding element $\alpha_s$ that
is a head.  This was the case when $s\geq1$. If we have a tail with $s=0$, we call this sequence
$\alpha_1,\dots, \alpha_m$ the \emph{initial tail}. Note that the initial tail lacks a head (i.e.\ $\alpha_0$ is not part of~\eqref{eq: yte4}), which is the reason why we will have to consider this case separately.

As a result, we can partition $(\alpha_i)_{i=1}^k$ into a
non-overlapping sequence of \emph{(i)}~the initial tail (possibly empty),
\emph{(ii)}~head-tail pairs, and \emph{(iii)}~elements larger than
$\frac{c}{L}$.  It is sufficient to show the bound for each part in our
partition.

\paragraph{Head-tail sequence.}
For a head-tail sequence $\alpha_s,\alpha_{s+1},\dots, \alpha_{s+m}$, we will
show that
\begin{equation*}
\sum_{i=s}^{s+m} \alpha_i \geq \frac{c(m+1)}{L}.
\end{equation*}
As we have already mentioned, for  $\alpha_{s+2}$ we have that
$\alpha_{s+2}  =\frac{\phi^2}{4\alpha_{s}L_{s+2}^2}\geq \frac{\phi^2}{4\alpha_s L^2}$. 
For elements $\alpha_{s+4},\dots, \alpha_{s+m}$ only the first option can occur,
since otherwise there  will be a contradiction with Claim 1. For $\alpha_{s+3}$,
however, both options are possible:
\[
\alpha_{s+3} = \incr \alpha_{s+2} \geq \frac{\gamma \phi^2}{4\alpha_s L^2}\quad  \text{ or }\quad \alpha_{s+3}=\frac{\phi^2}{4\alpha_{s+1} L_{s+3}^2} \geq \frac{\phi^2}{4\alpha_{s+1} L^2}.
\]
If  $\alpha_{s+3}=\incr \alpha_{s+2}\geq \frac{\gamma \phi^2}{4\alpha_s L^2}$, then 
\begin{align}\label{ineq-1}
  \sum_{i=s}^{s+m}{\alpha_i} & \geq \alpha_s + \alpha_{s+1} + \frac{\phi^2}{4\alpha_s L^2} + \frac{\incr \phi^2}{4\alpha_s L^2}+\dots + \frac{\incr^{m-2}\phi^2}{4\alpha_s L^2} \notag \\
                             &= \alpha_{s+1} + \alpha_s + \frac{\phi^2}{4\alpha_s L^2}(1+\dots + \incr^{m-2})\notag \\
                             &= \alpha_{s+1} + \alpha_s+ \frac{\phi^2}{4\alpha_s L^2} \frac{\gamma^{m-1}-1}{\gamma - 1}\notag \\
                             &\geq \alpha_{s+1} + \frac{\phi}{L}\sqrt{\frac{\gamma^{m-1}-1}{\gamma - 1}}\notag\\
                             & \geq \frac{c}{L} + \frac{mc}{L} = \frac{c(m+1)}{L},
\end{align}
where the last inequality follows from $\alpha_{s+1}\geq \frac{c}{L}$ and our choice of constant $c$.

If the second option is active for  $\alpha_{s+3}$, that is $\alpha_{s+3}\geq \frac{\phi^2}{4\alpha_{s+1} L^2}$, we have a  similar estimation to~\eqref{ineq-1}:
\begin{align}\label{ineq-2}
  \sum_{i=s}^{s+m}{\alpha_i} & \geq \alpha_s + \alpha_{s+1} +\alpha_{s+2} + \frac{\phi^2}{4\alpha_{s+1} L^2} + \frac{\incr \phi^2}{4\alpha_{s+1} L^2}+\dots + \frac{\incr^{m-3}\phi^2}{4\alpha_{s+1} L^2} \notag \\
                             &= ( \alpha_s + \alpha_{s+2}) + \alpha_{s+1} + \frac{\phi^2}{4\alpha_{s+1} L^2}(1+\dots + \incr^{m-3})\notag \\
                             &=  ( \alpha_s + \alpha_{s+2}) + \alpha_{s+1}+ \frac{\phi^2}{4\alpha_{s+1} L^2} \frac{\gamma^{m-2}-1}{\gamma - 1}\notag \\
                             &\geq ( \alpha_s + \alpha_{s+2}) + \frac{\phi}{L}\sqrt{\frac{\gamma^{m-2}-1}{\gamma - 1}}\notag\\
                             & \geq \frac{2c}{L} + \frac{c(m-1)}{L} =  \frac{c(m+1)}{L},
\end{align}
where the last inequality follows from
$\alpha_s + \alpha_{s+2}\geq \frac{\phi}{L}\geq \frac{2c}{L}$ and the definition
of $c$.

Hence, in both cases the desired bound holds.

\paragraph{Initial tail.}
For an empty initial tail there is nothing to prove. For a non-empty tail $\alpha_1,\dots, \alpha_m$, we will show that
\begin{equation*}
\sum_{i=1}^m \alpha_i \geq \frac{c(m-1)}{L}.
\end{equation*}
Since $\alpha_1\geq \frac{\phi}{2L}\geq \frac{c}{L}$, the first option cannot be active for
$\alpha_2$.  Hence, $\alpha_2 = \frac{\phi^2}{4\alpha_0L_2^2}\geq \frac{\phi^2}{4\alpha_0L^2}$.
First of all, note that for $m=1$ the desired inequality obviously holds:  $\sum_{i=1}^{1}
\alpha_i\geq 0$. For $m=2$, we have by Claim~2
\[\sum_{i=1}^{2}\alpha_i\geq \alpha_0 + \alpha_2 = \alpha_0 + \frac{\phi^2}{4\alpha_0L^2}\geq \frac{\phi}{L}\geq \frac{c}{L}.\]
Thus, for the rest of the proof we suppose that $m\geq 3$.

For steps $\alpha_4,\dots, \alpha_{m}$ only the first option can be active (by
Claim~1). For $\alpha_3$ we have two cases:
\begin{enumerate}
\item The first option for $\alpha_3$ is active, that is
  $\alpha_3 = \gamma \alpha_2$. Then similarly to \eqref{ineq-1} we have
\begin{align}\label{ineq-1_init}
  \sum_{i=1}^{m}{\alpha_i} & \geq \alpha_1 + \alpha_{2} + \gamma \alpha_2 +\dots + \gamma^{m - 2}\alpha_2\notag \\
                           &\geq \alpha_1 + \frac{\phi^2}{4\alpha_0 L^2}  \frac{\gamma^{m-1}-1}{\gamma - 1}\notag \\
                           &\geq \alpha_0 + \frac{\phi^2}{4\alpha_0 L^2}  \frac{\gamma^{m-1}-1}{\gamma - 1}\notag \\
                             &\geq \frac{\phi}{L}\sqrt{\frac{\gamma^{m-1}-1}{\gamma - 1}}\notag\\
                             & \geq \frac{c m}{L}.
\end{align}

\item 
The second option for $\alpha_3$ is active, that is
  $\alpha_3 = \frac{\phi^2}{4\alpha_1L_3^2}\geq \frac{\phi^2}{4\alpha_1L^2} $. Then similarly to \eqref{ineq-2}
  we have
\begin{align}\label{ineq-2_init2}
  \sum_{i=1}^{m}{\alpha_i} & = \alpha_1 + \alpha_{2} +\alpha_{3} +\gamma \alpha_3 +\dots + \gamma^{m-3}\alpha_3 \notag \\
                            &=  \alpha_1 + \alpha_{2}+\alpha_{3}(1+\dots + \incr^{m-3})\notag \\
                             &\geq   \alpha_1 + \alpha_{2} + \frac{\phi^2}{4\alpha_{1} L^2} \frac{\gamma^{m-2}-1}{\gamma - 1}\notag \\
                             &\geq  \alpha_2  + \frac{\phi}{L}\sqrt{\frac{\gamma^{m-2}-1}{\gamma - 1}}\notag\\
                             & \geq \frac{c(m-1)}{L},
\end{align}  
where we  used that $\alpha_2\geq 0$.
\end{enumerate}

Let us summarize what we have proved:
\begin{itemize}
  
\item the sequence $\alpha_1,\dots, \alpha_k$ can be divided into an initial part
  $\alpha_1,\dots \alpha_m$, some non-overlapping head-tails, and the remaining elements;
\item for every head-tail $\alpha_s,\dots, \alpha_{s+m}$ we showed that
  $\sum_{i=0}^m\alpha_{s+i}\geq \frac{c(m+1)}{L}$;
  
\item for the initial tail we showed that
  $\sum_{i=1}^m\alpha_{i}\geq \frac{c(m-1)}{L}$;
  
\item the remaining elements in $\alpha_1,\dots, \alpha_k$ are always greater or equal than
  $\frac c L$.
\end{itemize}
Hence, for each subsequence of length $l$, the sum of its elements is at least
$\frac{c(l-1)}{L}$. This allows us to conclude that
\[\sum_{i=1}^k \alpha_i \geq \frac{c(k-1)}{L}. \]
\end{proof}

\begin{theorem}\label{th: sec4}
  Let Assumption~\ref{as: as1}(i, iii, iv) hold and $F$ be locally
  Lipschitz. Let $\phi \in \left(1, \frac{1+\sqrt{5}}{2} \right)$,
  $\gamma = \frac{1}{\phi} + \frac{1}{\phi^2}$,
  $c = \min_{m\geq 2} \frac{\phi}{m} \sqrt{\frac{\gamma^{m-1}-1}{\gamma - 1}}> 0$, and
  $\zerg^k = \frac{1}{\sum_{i=1}^k\alpha_i}\sum_{i=1}^k \alpha_i\zb^i$. Then,
  Alg.~\ref{alg:agraal} with $\alpha_0$ as in~\eqref{eq:lns} has the rate
  \begin{equation*}
    \Gap\left( \zerg^k \right) \leq  \frac{LD}{2c(k-1)} .
  \end{equation*}
\end{theorem}
\begin{proof}
The result follows by using the definition of $\mathbf{Z}^k$ on the result of Lemma~\ref{lem: one_it_adap} and the lower bound of $\sum_{i=1}^k \alpha_k \geq \frac{(k-1)c}{L}$ derived in Lemma~\ref{lem: main_sec4}.
\end{proof}

The last thing left to understand is the value of $c$.
The next remark shows that it is large enough for a meaningful choice of parameters.
\begin{remark}\label{rem: rem_constants}
  Setting $\phi = 1.5$ as in the experiments of~\citep{malitsky2020golden} direct
  calculation gives $c \geq 0.5$.
\end{remark}

By proving a novel lower bound on the sum of the step sizes, we are able to prove a $\mathcal{O}(\frac1k)$ convergence rate without requiring an artificial upper bound $\bar{\alpha}$ on each individual step. Not only does this simplify the aGRAAL method, but also removes the spurious $L^2$ dependence from~\citep{malitsky2020golden}, showing that there is basically no extra cost of adaptivity.

\section{Nonmonotone problems with Weak Minty solutions}%
\label{sec:}
In Section~\ref{sec: 2raal}, for monotone problems, we observed empirically that
increasing $\phi$ leads to a certain speed-up, as in Fig.~\ref{fig:constant}. In
this section, we turn to a special class of \emph{nonmonotone} problems and show
that in some cases \emph{smaller} $\phi$ can be also favorable.

In particular, we consider the class of problems exhibiting a \emph{weak Minty} solution, which is given by a point $\zb^*$ such that
\begin{equation}\tag{WM}
  \label{eq:weak-minty}
  \langle F(\zb), \zb-\zb^* \rangle \ge -\frac{\rho}{2} \Vert F(\zb) \Vert^2 \quad \forall \zb \in \R^d.
\end{equation} 
\renewcommand*{\theHequation}{notag11.\theequation}%
This notion was recently introduced in~\citep{weak-minty-eg} in the context of von Neumann ratio games
--- a nonconvex-nonconcave min-max problem ---
and further investigated in~\citep{anchored_mgda_weak-minty,curvature-EG-weak-minty,axel-weak-minty-ogda}. While it is difficult to verify the existence of such solutions in practice, the template simultaneously captures different generalizations of monotonicity like the existence of a Minty solution~\citep{malitsky2020golden,optimistic-coherence} and negative comonotonicity~\citep{anchored-EG-neg-comonotone,negative-comonotone} in order to study nonconvex-nonconcave min-max problems. See~\citep{anchored-EG-neg-comonotone} for the implications between these different monotonicity concepts and~\citep{pp-eg-cohypomonotone} for a comparison in terms of last iterate convergence. To our knowledge, we give the first adaptive algorithm for this setting that does not require a linesearch at every iteration and only relies on \textit{local} Lipschitz continuity.

\subsection{Constant step size}
We first show the convergence of GRAAL with a constant step size.
\begin{theorem}\label{thm:graal-weak-minty}
  Let $F$ be $L$-Lipschitz and fulfill Assumption~\eqref{eq:weak-minty}, with $\rho<\frac{1}{L}$. Let $\phi>1$ be such that $\delta:= \frac{2-\phi}{\phi L} - \rho > 0$. Let $(\zb^k)$ and $(\bar{\zb}^k)$ be the iterates generated~\eqref{eq:graal} with $\alpha=\frac{2-\phi}{L}$. Then
  \begin{equation*}
     \min_{i=1,\dots, k} \Vert F(\zb^k) \Vert^2 \le \frac{L \phi }{k(2-\phi)(\phi-1)\delta } \left( \Vert \bar{\zb}^1 - \zb^* \Vert^2 + \frac{2}{\phi} \Vert \zb^1- \zb^0 \Vert^2 \right).
  \end{equation*}
\end{theorem}

In the limiting case when $\phi$ is close to $1$ we can allow for $\rho < \frac 1 L$ in Theorem~\ref{thm:graal-weak-minty}, which is precisely the best possible dependence for the (adaptive) EG$+$ method proven in~\citep{curvature-EG-weak-minty}.
For a more moderate choice, for example $\phi=\varphi=\frac{\sqrt 5 +1}{2}$, the bound on $\rho$ tightens to $\rho < \frac{1}{2L}$. See also Fig.~\ref{fig:polar-game} for empirical evidence of the need to reduce $\phi$ for problems with weak Minty solution.

\begin{figure}[ht]
  \begin{subfigure}[c]{.35\textwidth}
    \includegraphics[width=\linewidth]{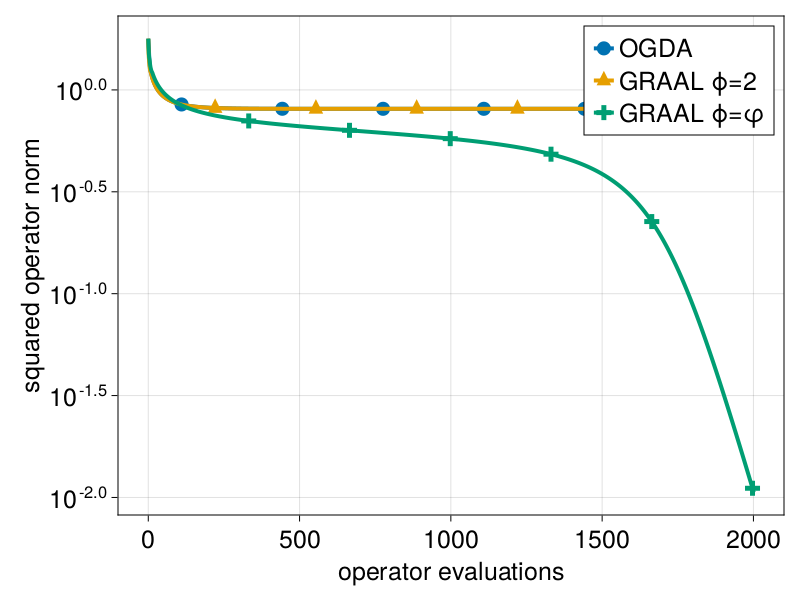}
  \end{subfigure}
  \hspace{.5cm}
  \begin{subfigure}[c]{.35\textwidth}
    \includegraphics[width=\linewidth]{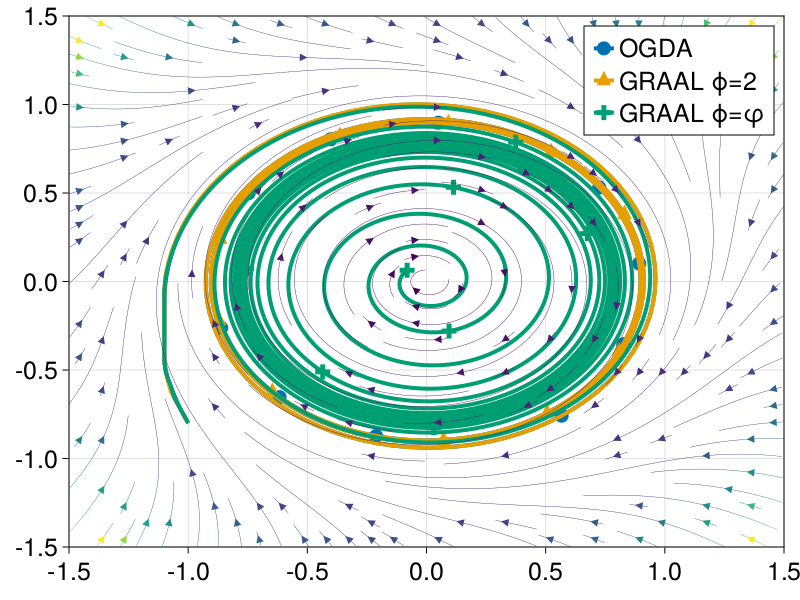}
  \end{subfigure}
  \centering
  \caption{\small A special parametrization of the Polar Game example from~\citep{curvature-EG-weak-minty}
    (see Section~\ref{sub:polar-game} for details), showing the need to reduce $\phi$ for nonmonotone problems with
    weak Minty solutions.\label{fig:polar-game}}
\end{figure}

Theorem~\ref{thm:graal-weak-minty} has the drawback of requiring knowledge of the weak Minty parameter $\rho$ to set $\phi$ such that $\delta > 0$.
In~\citep{curvature-EG-weak-minty} a similar problem was partially circumvented by an elaborate way to choose the corresponding parameter adaptively. Whether something similar can be done here is an open question. In practice, one can use a simple approach: run GRAAL with certain value of $\phi$; if it does not converge, decrease and try again.


\subsection{Adaptive step size}
We continue with the guarantees of aGRAAL under Assumption~\eqref{eq:weak-minty}.
\begin{theorem}\label{thm:agraal-weak-minty}
  Let $F$ be locally Lipschitz and fulfill Assumption~\eqref{eq:weak-minty}, and let $(\zb^k)$ be the sequence generated by~\eqref{eq:agraal}. As long as $\delta := \inf_k \alpha_k \Big(1 + \frac{1}{\phi} - \incr \phi\Big) - \rho > 0$
  we have
  \begin{equation*}
    \min_{i=1,\dots,k+1}  \Vert F(\zb^i) \Vert^2  \le \frac{D}{\delta \sum_{i=1}^{k+1} \alpha_i } \le \frac{c L D }{k \delta}, 
  \end{equation*}
  where $c$ denotes the constant given in Lemma~\ref{lem: main_sec4}.
\end{theorem}

\begin{remark}\label{rem:wm-lr}
  In particular, picking $\phi=1.1$ and $\incr = 1.1$ gives the condition $\rho < 0.69 \alpha_k$. As $\phi$ and $\incr$ get closer to $1$, we can allow for $\rho < \alpha_k$, which would precisely be the dependence between $\rho$ and the step size proven in~\citep{curvature-EG-weak-minty} for CurvatureEG$+$, a version of EG with linesearch.
\end{remark}

\begin{corollary}\label{cor:agraal-weak-minty-bounded-iter}
  In the setting of Theorem~\ref{thm:agraal-weak-minty}, if the iterates happen to remain bounded, we only require
  \begin{equation*}
    \inf_k \, \frac{\alpha_{k+1}}{\phi} + \alpha_k\Big( 1 + \frac{1}{\phi} - \incr\phi  \Big)- \rho  > 0,
  \end{equation*}
  to deduce a similar convergence statement, but with modified constant.
\end{corollary}

\begin{remark}
  The condition in Corollary~\ref{cor:agraal-weak-minty-bounded-iter} between the step sizes and the weak Minty parameter $\rho$ reduces, in the limiting case $\phi, \incr \to 1$, to $\alpha_{k+1}+\alpha_k > \rho$. This corresponds to doubling of the range of $\rho$ presented in Remark~\ref{rem:wm-lr}.
\end{remark}

\begin{remark}[Lower bounding $\alpha_k$.]
In this section we relaxed the condition between the Lipschitz constant $L$ and parameter $\rho$ in~\eqref{eq:weak-minty} to a condition between $\alpha_k$ and $\rho$. We do so in the hope that aGRAAL is able to take larger steps, since they are based on the local Lipschitz constants. Unfortunately, we cannot give a good lower bound on any individual $\alpha_k$ in general. We see from the experiments that the step size sometimes does become lower than the one chosen by linesearch (see Fig.~\ref{fig:forsaken}). This seems to, however, not harm the performance of the algorithm.
\end{remark}

\subsection{Experiments}
As discussed in the main section of the paper, in order to solve weak Minty problems, for all known methods two things are important: \emph{(i)} the method needs to be modified in a way that could be interpreted as making it more ``conservative''; \emph{(ii)} step sizes should be large.

In the case of the golden ratio algorithm, the former means averaging with more of the (old) $\bar{\zb}^{k-1}$ iterate, i.e.\ decreasing $\phi$, see Fig.~\ref{fig:forsaken}.  For EG this means reducing the step size in the update of the $\zb^k$ sequence, as proposed in~\citep{weak-minty-eg,curvature-EG-weak-minty}. For the constant step size methods in question it seems like this is all we can do. However, by choosing the step sizes adaptively, be it via linesearch as for the CurvatureEG$+$ method from~\citep{curvature-EG-weak-minty}, or by~\eqref{eq: ss_def} in the case of~\eqref{eq:agraal}, we can hope to take steps larger than the global Lipschitz constant would allow.

\subsubsection{Forsaken}%

In Fig.~\ref{fig:forsaken} the following problem formulation is used, which originated in~\cite[Example 5.2]{limits-of-minimax} as a particularly difficult instance of min-max under the name of ``Forsaken'':
\begin{equation*}
  \min_{x\in \R} \, \max_{y\in \R} \, x(y-0.45) + f(x) - f(y),
\end{equation*}
where $f(z) = \frac{1}{4}z^2 - \frac12 z^4 + \frac16 z^6$. This problem exhibits a solution at $(x^*,y^*)\approx (0.08, 0.4)$, but also \emph{two} limit cycles not containing any critical point of the objective function.

\begin{figure}[ht]
  \begin{subfigure}[c]{.3\textwidth}
    \includegraphics[width=\linewidth]{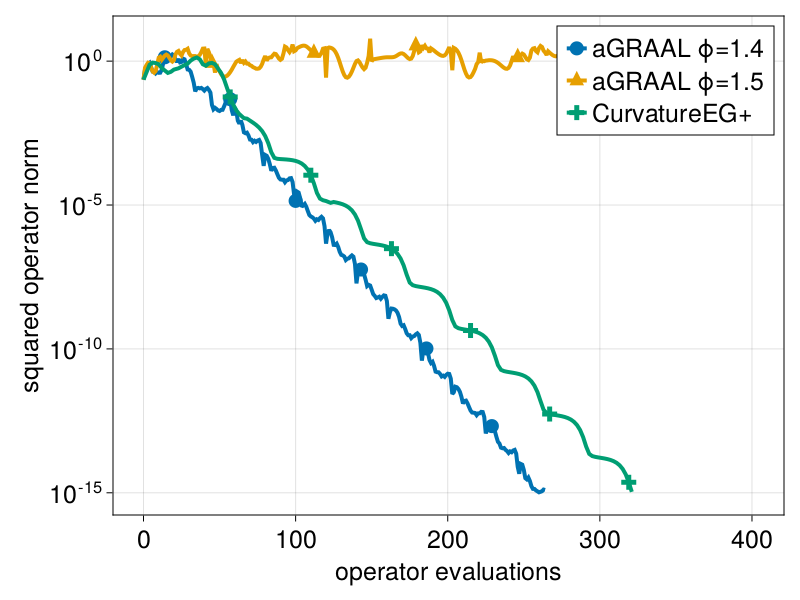}
  \end{subfigure}
  \hspace{.3cm}
  \begin{subfigure}[c]{.3\textwidth}
    \includegraphics[width=\linewidth]{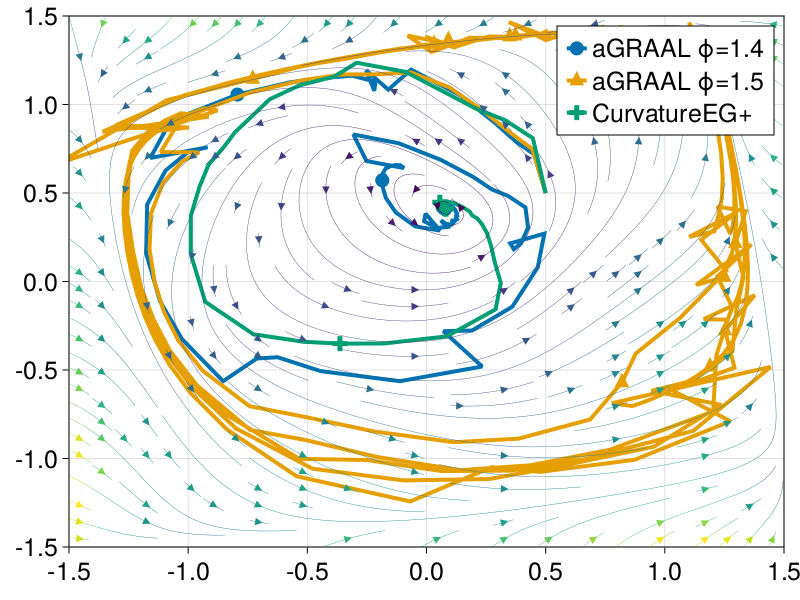}
  \end{subfigure}
  \hspace{.3cm}
  \begin{subfigure}[c]{.3\textwidth}
    \includegraphics[width=\linewidth]{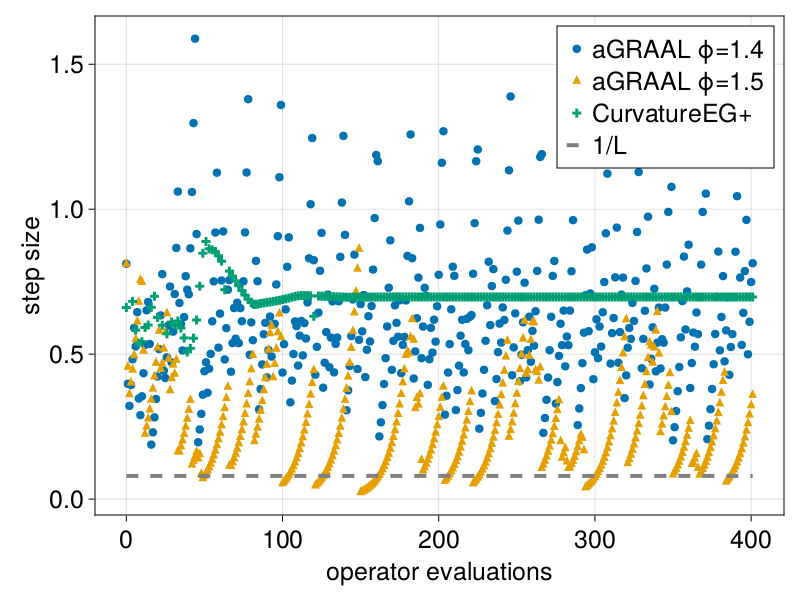}
  \end{subfigure}
  \centering
  \caption{{\small The ``Forsaken'' example of~\cite[Example 5.2]{limits-of-minimax}, which is further explored in~\citep{curvature-EG-weak-minty}. If $\phi$ is too large in aGRAAL, the method gets stuck in the limit cycle. Only for small $\phi$ the method converges, and does so rapidly.\label{fig:forsaken}}}
\end{figure}

\subsubsection{Polar Game}%
\label{sub:polar-game}

Fig.~\ref{fig:polar-game} and~\ref{fig:polar-game-adaptive} display results on
the so-called Polar Game introduced in~\cite[Example 3]{curvature-EG-weak-minty}
using the following parametrization for $a>0$
\begin{equation*}
  F(x,y) = (\psi(x,y)-y, \psi(y,x)+x),
\end{equation*}
where $\psi(x,y) =  a \frac{1}{4} x (-1+x^2+y^2)(-1 + 4x^2 + 4y^2)$. The problem exhibits a limit cycle attracting solutions away from the solution in the center.

\begin{figure}[ht]
  \begin{subfigure}[c]{.3\textwidth}
    \includegraphics[width=\linewidth]{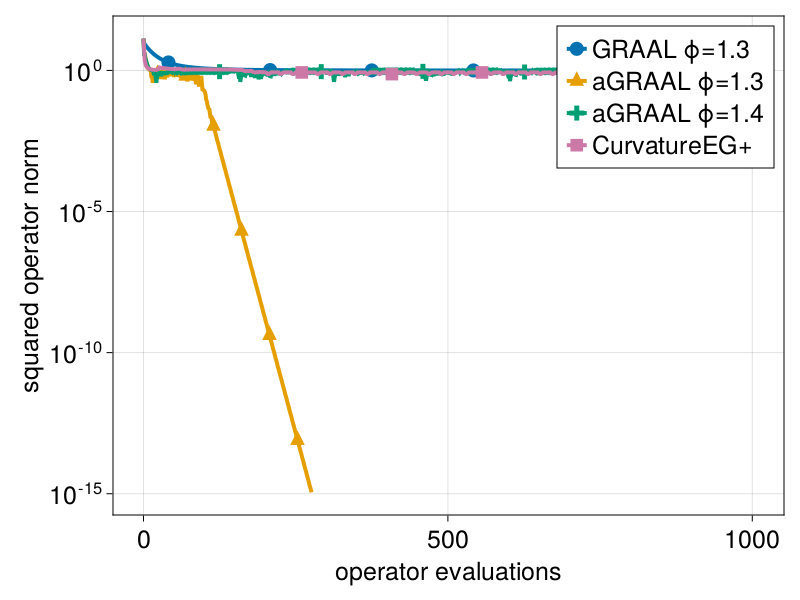}
  \end{subfigure}
  \hspace{.3cm}
  \begin{subfigure}[c]{.3\textwidth}
    \includegraphics[width=\linewidth]{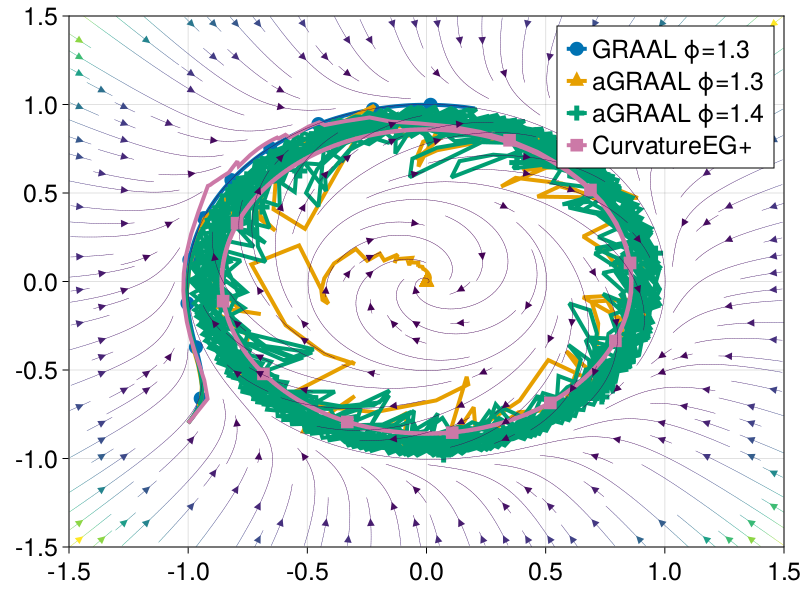}
  \end{subfigure}
  \hspace{.3cm}
  \begin{subfigure}[c]{.3\textwidth}
    \includegraphics[width=\linewidth]{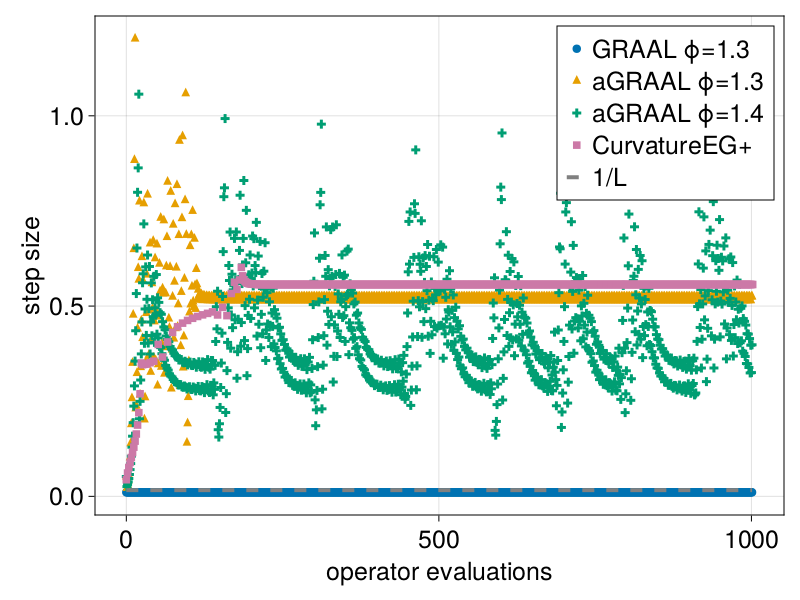}
  \end{subfigure}
  \centering
  \caption{\small{A different parametrization of the Polar Game, $a=3$, showing the importance of adaptive step sizes. We observe that~\eqref{eq:agraal} is the only method able to converge.\label{fig:polar-game-adaptive}}}
\end{figure}

Fig.~\ref{fig:polar-game-adaptive} shows that reducing $\phi$ alone will not be enough for problems where $\rho > \frac{1}{L}$ is not satisfied (the blue line does not converge). One has to choose larger step sizes.

\subsubsection{A lower bound~\citep[Example 5]{curvature-EG-weak-minty}}

In~\cite{curvature-EG-weak-minty} it was shown that EG$+$ (with arbitrary small second step size) may diverge if $\rho > \frac{1}{L}$ via the following unconstrained min-max problems
\begin{equation*}
  \min_{x\in \R} \, \max_{y\in \R} \, a xy + \frac{b}{2}(x^2-y^2),
\end{equation*}
for $a>0$ and $b<0$. The associated operator $F$, simply given by
\begin{equation*}
  F(x,y) = (ay + bx, by-ax)
\end{equation*}
\begin{figure}[ht]
  \begin{subfigure}[c]{.3\textwidth}
    \includegraphics[width=\linewidth]{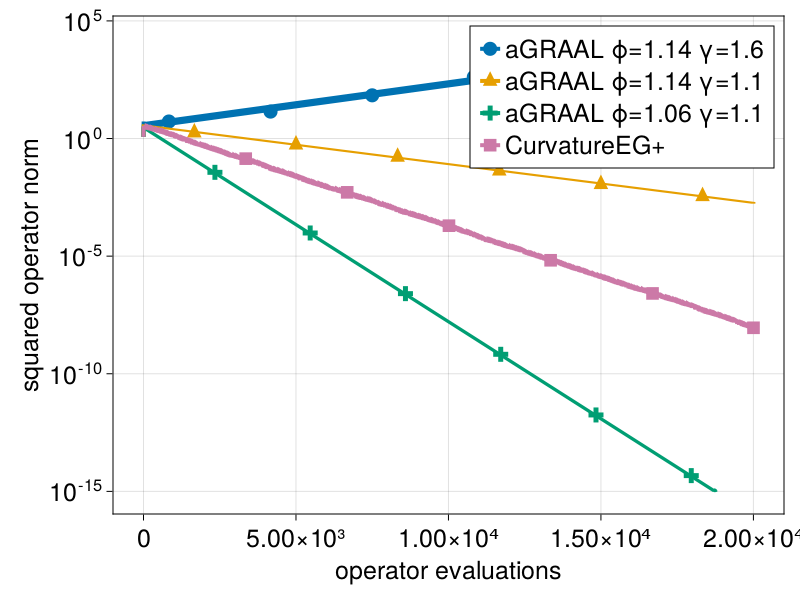}
  \end{subfigure}
  \hspace{.5cm}
  \begin{subfigure}[c]{.3\textwidth}
    \includegraphics[width=\linewidth]{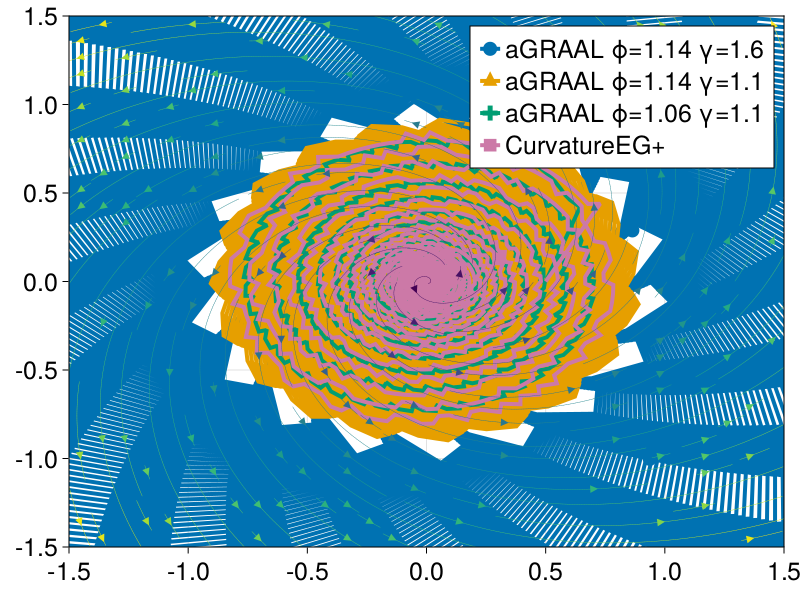}
  \end{subfigure}
  \hspace{.5cm}
  \begin{subfigure}[c]{.3\textwidth}
    \includegraphics[width=\linewidth]{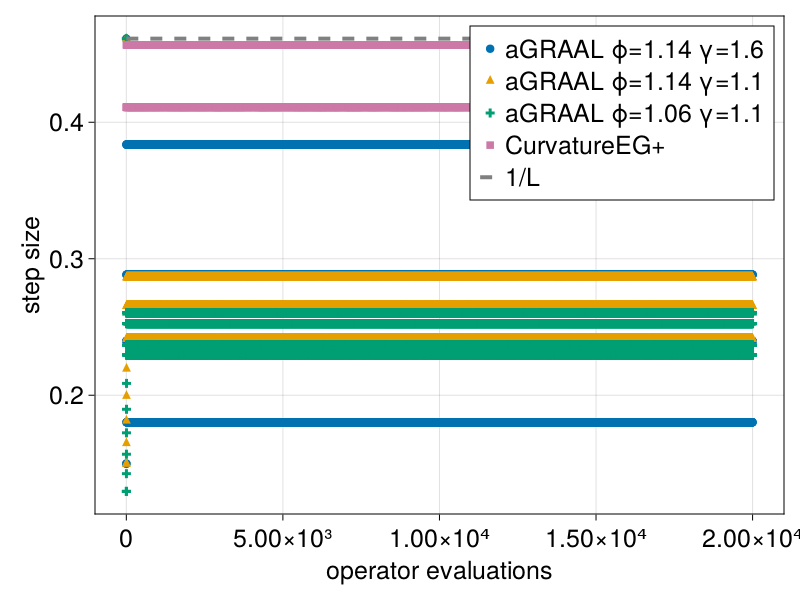}
  \end{subfigure}
  \centering
  \caption{\small{A special parametrization ($a^2=3.7, b=-1$) of~\cite[Example 5]{curvature-EG-weak-minty}, illustrating the fact that sometimes $\gamma$ does indeed need to be chosen smaller as predicted by theory. \label{fig:lower-bound-adaptive}}}
\end{figure}
 is Lipschitz with $L=\sqrt{a^2+b^2}$ and $(0, 0)$ is a weak Minty solution with constant $\rho = -\frac{2 b}{a^2+b^2}$.

In Fig.~\ref{fig:lower-bound-adaptive} we see that, as predicted by Theorem~\ref{thm:agraal-weak-minty} and Corollary~\ref{cor:agraal-weak-minty-bounded-iter}, not only reducing $\phi$ but also decreasing $\gamma$ can prevent divergence. At first glance, this is somewhat surprising as larger $\gamma$ allows for a faster increase of the step size from one iteration to the next, and we have already seen that larger step sizes are important for convergence. We suspect this to be rooted in the special step size rule of aGRAAL, see~\eqref{eq: ss_def}, where one large step can lead to decrease in the step size of the next iteration.

\acks{The work of Axel B\"ohm was funded by the Austrian Science Fund (FWF): W1260-N35.
The work of Ahmet Alacaoglu was funded by NSF awards 2023239 and 2224213; and DOE
ASCR Subcontract 8F-30039 from Argonne National Laboratory. The work of Yura
Malitsky was supported by the Wallenberg Al, Autonomous Systems and Software
Program funded by the Knut and Alice Wallenberg Foundation, No.~305286.
}

\newpage

\appendix
\section{Details on Section~\ref{sec: 2raal}}\label{sec:appendix}
For convenience, we give a full description of the algorithm
 \begin{algorithm}[H]
   \caption{GRAAL}\label{alg:graal}
   \algorithmicrequire{ $\bzb^{-1} = \zb^0 \in C$, $\phi \in (1, 2]$.} 
   \begin{algorithmic}[1]
     \For{$k\geq0$}
        \State  $\bzb^k = \frac{\phi-1}{\phi} \zb^k + \frac{1}{\phi}\bzb^{k-1}$
        \State $\zb^{k+1} = P_C(\bzb^k - \alpha F(\zb^k))$
     \EndFor{}
   \end{algorithmic}
 \end{algorithm}

\subsection{Proof of Corollary~\ref{cor: as3}}
We start with the following lemma which essentially summarizes the existing result of~\citep{malitsky2020golden} for one iteration of the algorithm.
Letting $\phi=2$ and $\varepsilon=0$ in this lemma gives~\eqref{eq: fe2} which is the starting point of Theorem~\ref{thm:bounded}.
\begin{lemma}[based on Theorem 2 in~\citep{malitsky2020golden}]\label{lem: one_it}
Let Assumption~\ref{as: as1} hold and $\alpha = \frac{\phi-2\varepsilon}{2L}$ for any $\varepsilon\in[0, \phi/2)$ in~\ref{eq:graal}. Then, for any $\zb\in C$
\begin{multline*}
G(\zb^k, \zb)+\frac{\phi}{\phi-1} \| \bzb^{k+1} - \zb\|^2 + \frac{\phi}{2} \| \zb^{k+1} - \zb^{k}\|^2 + \left( \phi - 1 - \frac{1}{\phi} \right) \| \zb^{k+1} - \bzb^k \|^2 \\
\leq \frac{\phi}{\phi-1} \| \bzb^{k} - \zb\|^2 + \frac{(\phi-2\varepsilon)^2}{2\phi} \| \zb^{k} - \zb^{k-1}\|^2 - \frac{1}{\phi} \| \zb^{k} - \bzb^{k-1} \|^2.
\end{multline*}
\end{lemma}
\begin{proof}[of Lemma~\ref{lem: one_it}]
By applying prox-inequality to the definition of $\zb^{k+1}$, we have for any $\zb\in C$ that
  \begin{equation}
    \label{eq:11_a}
    \langle \zb^{k+1} - \bar{\zb}^k + \alpha F(\zb^k), \zb - \zb^{k+1} \rangle \geq 0.
  \end{equation}
  Similarly, with the definition of $\zb^k$, we have
  \begin{equation}
    \label{eq:12_a}
    \langle \zb^k - \bar{\zb}^{k-1} + \alpha F(\zb^{k-1}), \zb^{k+1} - \zb^k \rangle \geq 0.
  \end{equation}
  Since $\zb^k-\bar{\zb}^{k-1} = \frac{1+\phi}{\phi}(\zb^k-\bar{\zb}^k)= \phi (\zb^k-\bar{\zb}^k)$, we can rewrite~\eqref{eq:12_a} as
  \begin{equation}
    \label{eq:13_a}
    \langle \phi (\zb^k-\bar{\zb}^k)+ \alpha F(\zb^{k-1}), \zb^{k+1}-\zb^k \rangle \geq 0.
  \end{equation}
  Summing up~\eqref{eq:11_a} and~\eqref{eq:13_a} yields
  \begin{multline}
    \label{eq:14_a}
    \langle \zb^{k+1}-\bar{\zb}^k, \zb - \zb^{k+1} \rangle + \phi \langle \zb^k-\bar{\zb}^k, \zb^{k+1}-\zb^k \rangle + \alpha \langle F(\zb^k), \zb -\zb^k \rangle \\
    + \alpha \langle F(\zb^k)-F(\zb^{k-1}), \zb^k-\zb^{k+1} \rangle \geq 0.
  \end{multline}
Expressing the first two terms in~\eqref{eq:14_a} by norms and using the definition of $G(\zb^k, \zb)$ from~\eqref{eq: we3}, we deduce
  \begin{multline}
    \label{eq:15_a}
G(\zb^k, \zb)+    \Vert \zb^{k+1}-\zb \Vert^2 \le \Vert \bar{\zb}^k-\zb \Vert^2 - \Vert \zb^{k+1}-\bar{\zb}^k \Vert^2 + 2 \alpha \langle F(\zb^k)-F(\zb^{k-1}), \zb^k-\zb^{k+1} \rangle \\
    + \phi (\Vert \zb^{k+1}-\bar{\zb}^k \Vert^2 - \Vert \zb^{k+1}-\zb^k \Vert^2 - \Vert \zb^k-\bar{\zb}^k \Vert^2).
  \end{multline}
The definition of $\bar{\zb}^k$ gives
  \begin{align}
    \Vert \zb^{k+1}-\zb \Vert^2 &= \frac{\phi}{\phi-1} \Vert \bar{\zb}^{k+1}-\zb \Vert^2 - \frac{1}{\phi-1} \Vert \bar{\zb}^k-\zb \Vert^2 + \frac{\phi}{(\phi-1)^2} \Vert \bzb^{k+1}-\bar{\zb}^k \Vert^2 \notag \\
    &= \frac{\phi}{\phi-1} \Vert \bar{\zb}^{k+1}-\zb \Vert^2 - \frac{1}{\phi-1} \Vert \bar{\zb}^k-\zb \Vert^2 + \frac{1}{\phi} \Vert \zb^{k+1}-\bar{\zb}^k \Vert^2,    \label{eq:17_a}
  \end{align}
  where the second equality used $\bzb^{k+1} - \bzb^k = \frac{\phi-1}{\phi}(\zb^{k+1} - \bzb^k)$.
  
  Additionally, from $\alpha = \frac{\phi-2\varepsilon}{2L}$ and Young's inequality, we have
  \begin{equation}
    \label{eq:19_a}
    \begin{aligned}
      2 \alpha \langle F(\zb^k)-F(\zb^{k-1}), \zb^k-\zb^{k+1} \rangle &\le (\phi-2\varepsilon) \Vert F(\zb^k)-F(\zb^{k-1}) \Vert \Vert \zb^{k+1}-\zb^k \Vert \\
      &\le \frac{\phi}{2} \Vert \zb^{k+1}-\zb^k \Vert^2 + \frac{(\phi-2\varepsilon)^2}{2\phi} \Vert \zb^k-\zb^{k-1} \Vert^2.
    \end{aligned}
  \end{equation}
  Using~\eqref{eq:17_a},~\eqref{eq:19_a}, and $\phi^2\|\zb^k - \bzb^{k} \|^2 = \| \zb^k - \bzb^{k-1}\|^2$ on~\eqref{eq:15_a} gives the desired conclusion.
\end{proof}

\begin{lemma}[Upper bound of $R^2$ used in Theorem~\ref{thm:bounded}]\label{lem: rto4}
Let Assumption~\ref{as: as1} hold. Then for the first iteration of Alg.~\ref{alg:graal}, we have that
\begin{equation*}
R^2 = \| \zb^1 - \zb^* \|^2 + \| \zb^0 - \zb^* \|^2 \leq 3 \|\zb^0 - \zb^*\|^2.
\end{equation*}
\end{lemma}
\begin{proof}
Recall that  $\zb^{1} = P_C(\zb^0 - \alpha F(\zb^0))$ and $\zb^* = P_C(\zb^* - \alpha F(\zb^*))$. The former is because $\bzb^0 = \zb^0$ by the initialization in Alg.~\ref{alg:graal}. The latter follows by the definition of $\zb^*$ as a solution of~\eqref{eq: vi_prob}. By firm-nonexpansiveness of $P_C$, we have
\begin{align*}
  \|\zb^1 - \zb^*\|^2 &\leq \|(\zb^0-\alpha F(\zb^0)) - (\zb^* - \alpha F(\zb^*)) \|^2 - \|\zb^0-\alpha F(\zb^0) - \zb^1 + \alpha F(\zb^*) \|^2.
\end{align*}
Expanding the terms in the right-hand side we have that
\begin{multline}\label{eq: clg3}
\|\zb^1 - \zb^*\|^2 \leq \| \zb^0-\zb^*\|^2 - \| \zb^1-\zb^0\|^2 + 2\alpha \langle F(\zb^0) - F(\zb^*), \zb^0 - \zb^1 \rangle \\
-2\alpha \langle F(\zb^0)-F(\zb^*), \zb^0-\zb^*\rangle.
\end{multline}
By monotonicity, we have $\langle F(\zb^0)-F(\zb^*), \zb^0-\zb^*\rangle \geq 0$. By using $\alpha \leq \frac{1}{L}$, we have
\begin{equation*}
2\alpha \langle F(\zb^0) - F(\zb^*), \zb^0 - \zb^1 \rangle \leq 2\alpha L\|\zb^0 - \zb^* \| \|\zb^0 - \zb^1\| \leq \|\zb^0 - \zb^* \|^2+ \|\zb^0 - \zb^1\|^2.
\end{equation*}
Combining the last two estimates for the inner products in~\eqref{eq: clg3}, we have
\begin{align*}
\|\zb^1 - \zb^*\|^2 \leq 2\| \zb^0-\zb^*\|^2,
\end{align*}
which gives the result.
\end{proof}

\subsection{Alternative proof of Theorem~\ref{thm:bounded} via SDP} The proof of boundedness of $(\zb^k)$ presented in the main part was not very standard. In this section, we  provide an alternative proof of that fact which is based on a semidefinite program combined with induction. This approach is easily generalizable to the case of $\phi < 2$ even though we give it for $\phi=2$ for simplicity.

As before, we assume that $\n{\bzb^i-\zb^*}^2\leq 2R^2$ for all $i=1,\dots,k$ and we must show that $\n{\bzb^{k+1}-\zb^*}^2\leq 2R^2$. Our main tool will be inequality \eqref{beauty_main} which holds for all $k\geq 1$ and which we recall here 
\begin{align}
2\n{\bzb^{k+1}-\zb^*}^2 + \n{\zb^{k+1}-\zb^k}^2
&\leq  \| \zb^1 - \zb^*\|^2+\| \zb^0 - \zb^*\|^2  +\frac 1 2\n{\zb^{k+1}-\bzb^k}^2\label{beauty_main_app}
\end{align}
Without loss of generality, we assume that $\zb^* = 0$ and that $R^2 = \| \zb^1 - \zb^*\|^2+\| \zb^0 - \zb^*\|^2 \leq 1$. The first assumption is valid because we can always redefine sequences  as $\zb^k:=\zb^k-\zb^*$ and  $\bzb^k:=\bzb^k-\zb^*$, while the second~--- because we can rescale the norm $\|\cdot\|$ by any positive number.

Using these two assumptions and that $\zb^{k} = 2\bzb^k-\bzb^{k-1}$ we can rewrite \eqref{beauty_main_app} as
\begin{align*}
2\n{\bzb^{k+1}}^2 + \n{2\bzb^{k+1}-\bzb^k -2\bzb^k+\bzb^{k-1}}^2
  &\leq  1  +\frac 1 2\n{2\bzb^{k+1}-\bzb^k-\bzb^k}^2,
\end{align*}
or 
\begin{align}
 \label{beauty_main_app2}
2\n{\bzb^{k+1}}^2 + \n{2\bzb^{k+1}-3\bzb^k+\bzb^{k-1}}^2  &\leq  1  +2\n{\bzb^{k+1}-\bzb^k}^2.
\end{align}
The latter inequality is quadratic in $\bzb^{k+1},\bzb^k, \bzb^{k-1}$. Hence, after expanding the norms, we can rewrite it in a matrix notation as
\begin{align*}
  \tr(\Gb\cdot \Mb)\leq 1,
\end{align*}
where $\Gb$ is a Gram matrix 
\[\Gb =
\begin{bmatrix}
\lr{\bzb^{k+1}, \bzb^{k+1}} & \lr{\bzb^{k+1}, \bzb^{k}} & \lr{\bzb^{k+1}, \bzb^{k-1}} \\
\lr{\bzb^{k}, \bzb^{k+1}} & \lr{\bzb^{k}, \bzb^{k}} & \lr{\bzb^{k}, \bzb^{k-1}} \\
\lr{\bzb^{k-1}, \bzb^{k+1}} & \lr{\bzb^{k-1}, \bzb^{k}} & \lr{\bzb^{k-1}, \bzb^{k-1}} 
\end{bmatrix}
\qquad \text{and}\qquad
\Mb = 
\begin{bmatrix}
4 & -4 & 2\\
-4 & 7 & -3\\
2 & -3 & 1 
\end{bmatrix}. 
\]
By induction assumption we have that $\Gb_{22}\leq 2$ and $\Gb_{33}\leq 2$ and we must show that $\Gb_{11}\leq 2$. Consider the following semidefinite program
\begin{align*}
  \max_{\Gb}\, \Gb_{11} \quad &\text{subject to}\\
  \Gb &\succcurlyeq 0\\
  \tr(\Gb\cdot \Mb)&\leq 1\\
  \Gb_{22}&\leq 2\\
  \Gb_{33}&\leq 2
\end{align*}
If we can show that its optimal value is less than $2$, then we are done: it will automatically imply that $\n{\bzb^{k+1}}^2\leq 2$. Now, by solving it, we obtain $\Gb_{11}\approx 1.49259 $.

Therefore, we have proved that for all $k$, $\n{\bzb^k-\zb^*}^2 \leq 2R^2$.

\begin{remark}
  The semidefinite program actually allows us to show a slightly tighter bound: $\n{\bzb^{k}-\zb^*}^2 \leq 1.2 R^2$, but we kept the constant $2$ for consistency with the previous approach.
\end{remark}

\section{Details for Section~\ref{sub:adap-graal-no-hyper}}

\subsection{Implementation of the linesearch}%
\label{sub:}

In~\eqref{eq:lns} we outline a very particular linesearch which decreases $\alpha_0$ by $\gamma$ in every step. Since the canonical choice of $\phi$ proposed in~\citep{malitsky2020golden} yields the small $\gamma=1.1$ this could result in many backtracking iterations if our initial guess is bad.
For practical implementation, therefore, it is better to use first a coarse
reduction of $\alpha_0$, say with a factor of $10$ and only at the end to switch to a
fine factor $\gamma$.

\section{Weak Minty proofs}%
\label{sec:append-minty}

\subsection{Constant step size}%
\label{sub:}

By following the line of reasoning as in the monotone case one can only derive $\rho<\frac{1}{2L}$. We do not present the proof here but it is only a simple modification similar to the one we show later for the adaptive step size. In order to derive the $\rho< \frac1L$ bound (matching EG) one has to take a completely different approach. The high-level reason is that the monotone analysis requires $\alpha \le \frac{\phi}{2L}$. This is counter intuitive as a smaller $\phi$ makes~\eqref{eq:graal} more conservative since the iterates are more anchored to $\bar{\zb}^k$, and at the same time requires a smaller step size. A more conservative averaging of the iterates should allow for a more aggressive step size, which is precisely the behavior we observe here.
For convenience let $\gb^k = F(\zb^k)$.

\begin{lemma}%
  \label{lem:}
  Let $F$ be $L$-Lipschitz and fulfill Assumption~\eqref{eq:weak-minty}. Let $(\zb^k)$ and $(\bar{\zb}^k)$ be the iterates generated~\eqref{eq:graal}. Then
  \begin{multline}
    \label{eq:descent}
    \Vert \bar{\zb}^{k+1}-\zb^* \Vert^2 \le \Vert \bar{\zb}^k-\zb^* \Vert^2 + \frac{\phi-1}{\phi}\alpha \rho \Vert \gb^k \Vert^2  +  \frac{4}{\phi (\phi-1)}\alpha \langle \gb^k - \gb^{k-1}, \zb^k-\zb^{k+1} \rangle\\
  - \frac{3-\phi}{\phi-1} \Vert \zb^{k+1}-\zb^k \Vert^2 - \frac{\alpha^2(\phi+1)}{\phi^2(\phi-1)} \Vert \gb^k - \gb^{k-1} \Vert^2.
  \end{multline}
\end{lemma}
\begin{proof}
  First we observe simply from the definition of the iterates
  \begin{equation}
    \label{eq:firsty}
    \begin{aligned}
      \Vert \bar{\zb}^{k+1}-\zb^* \Vert^2 &= \left\Vert \frac{\phi-1}{\phi} \zb^{k+1} + \frac1\phi \bar{\zb}^k - \zb^* \right\Vert^2 \\
                              &= \left\Vert \frac{\phi-1}{\phi}(\bar{\zb}^k-\alpha \gb^k) + \frac1\phi \bar{\zb}^k-\zb^* \right\Vert^2 \\
                              &= \left\Vert \bar{\zb}^k - \frac{\phi-1}{\phi}\alpha \gb^k-\zb^* \right\Vert^2 \\
                              &= \Vert \bar{\zb}^k -\zb^* \Vert^2 -2 \frac{\phi-1}{\phi} \alpha \langle \gb^k, \bar{\zb}^k-\zb^* \rangle + \left\Vert \frac{\phi-1}{\phi}\alpha \gb^k \right\Vert^2 \\
                              &\le \Vert \bar{\zb}^k - \zb^* \Vert^2 + \frac{\phi-1}{\phi}\alpha \rho \Vert \gb^k \Vert^2 - 2 \frac{\phi-1}{\phi}\alpha \langle \gb^k, \bar{\zb}^k-\zb^k \rangle + \left\Vert \frac{\phi-1}{\phi}\alpha \gb^k \right\Vert^2 \\
                              &= \Vert \bar{\zb}^k - \zb^* \Vert^2 + \frac{\phi-1}{\phi}\alpha \rho \Vert \gb^k \Vert^2 - 2 \frac{\phi-1}{\phi}\alpha \langle \gb^k, \frac1\phi \alpha \gb^{k-1} \rangle + \left\Vert \frac{\phi-1}{\phi}\alpha \gb^k \right\Vert^2\\
                              &= \Vert \bar{\zb}^k - \zb^* \Vert^2 + \frac{\phi-1}{\phi}\alpha \rho \Vert \gb^k \Vert^2 - \frac{\phi-1}{\phi}\alpha^2 \left\langle \gb^k, \frac{2}{\phi}\gb^{k-1} - \frac{\phi-1}{\phi}\gb^k \right\rangle,
    \end{aligned}
  \end{equation}
where we used~\eqref{eq:weak-minty} to deduce the inequality.
Now we want to prove the following equality
\begin{multline}
  \label{eq:crazy}
  - \frac{\phi-1}{\phi}\alpha^2 \left\langle \gb^k, \frac{2}{\phi}\gb^{k-1} - \frac{\phi-1}{\phi}\gb^k \right\rangle = \frac{4}{\phi (\phi-1)}\alpha \langle \gb^k - \gb^{k-1}, \zb^k-\zb^{k+1} \rangle\\
  - \frac{3-\phi}{\phi-1} \Vert \zb^{k+1}-\zb^k \Vert^2 - \frac{\alpha^2(\phi+1)}{\phi^2(\phi-1)} \Vert \gb^k - \gb^{k-1} \Vert^2.
\end{multline}
We can verify this by expanding all iterates in terms of operators. Observe first that
\begin{equation}
  \label{eq:iterates-as-grad}
  \zb^{k+1}-\zb^k = \bar{\zb}^k - \alpha \gb^k - \zb^k = \frac1\phi (\bar{\zb}^{k-1}-\zb^k)- \alpha \gb^k =  \frac1\phi \alpha \gb^{k-1} - \alpha \gb^k = \alpha \frac1\phi (\gb^{k-1} - \gb^k) - \alpha\frac{\phi-1}{\phi}\gb^k.
\end{equation}
Now we use this to deduce
\begin{align*}
  - \Vert \zb^{k+1}-\zb^k \Vert^2 &= - \alpha^2 \left\Vert \frac1\phi (\gb^{k-1} - \gb^k) - \frac{\phi-1}{\phi}\gb^k \right\Vert^2 \\
  &= - \frac{\alpha^2}{\phi^2} \Vert \gb^{k-1} - \gb^k \Vert^2 + \frac{\phi-1}{\phi^2}\alpha^2 \langle \gb^k, \gb^{k-1}-\gb^k \rangle - \alpha^2 {\left(\frac{\phi-1}{\phi}\right)}^2 \Vert \gb^k \Vert^2 \\
  &= - \frac{\alpha^2}{\phi^2} \Vert \gb^{k-1} - \gb^k \Vert^2 + \frac{\phi-1}{\phi^2} \alpha^2 \langle \gb^k, 2 \gb^{k-1} - 2 \gb^k - (\phi-1) \gb^k \rangle \\
  &= - \frac{\alpha^2}{\phi^2} \Vert \gb^k - \gb^{k-1} \Vert^2 + \frac{\phi-1}{\phi^2}\alpha^2 \left\langle \gb^k, 2\gb^{k-1}- (\phi+1)\gb^k \right\rangle.
\end{align*}
Therefore, by multiplying both sides by $\frac{3-\phi}{\phi-1}$ we get
\begin{equation}
  \label{eq:diff-iterates}
  -\frac{3-\phi}{\phi-1} \Vert \zb^{k+1}-\zb^k \Vert^2 = - \frac{3-\phi}{\phi-1} \frac{\alpha^2}{\phi^2} \Vert \gb^k - \gb^{k-1} \Vert^2 + \frac{3-\phi}{\phi^2} \alpha^2 \left\langle \gb^k, 2\gb^{k-1}- (\phi+1)\gb^k \right\rangle.
\end{equation}
Again, by going to $\gb^k$ everywhere we have
\begin{align*}
  \frac{\alpha}{\phi} \langle \gb^k - \gb^{k-1}, \zb^k - \zb^{k+1} \rangle &= \frac{\alpha^2}{\phi} \langle \gb^k - \gb^{k-1}, \gb^k - \frac1\phi \gb^{k-1} \rangle \\
  &= \frac{\alpha^2}{\phi} \langle \gb^k - \gb^{k-1}, \frac1\phi (\gb^k - \gb^{k-1}) + (1-\frac1\phi)\gb^k \rangle \\
  &= \frac{\alpha^2}{\phi^2} \Vert \gb^{k-1}-\gb^k \Vert^2 + \frac{\phi-1}{\phi^2} \alpha^2 \langle \gb^k , \gb^k - \gb^{k-1} \rangle,
\end{align*}
and therefore by multiplying both sides by $\frac{4}{\phi-1}$
\begin{align}
  \label{eq:inner-product}
  \frac{4}{\phi-1}\frac{\alpha}{\phi} \langle \gb^k - \gb^{k-1}, \zb^k - \zb^{k+1} \rangle = \frac{4}{\phi-1}\frac{\alpha^2}{\phi^2} \Vert \gb^{k-1}-\gb^k \Vert^2 + \frac{4}{\phi^2} \alpha^2 \langle \gb^k , \gb^k - \gb^{k-1} \rangle.
\end{align}
By combining~\eqref{eq:inner-product} and~\eqref{eq:diff-iterates} we deduce
\begin{multline*}
  - \frac{\phi-1}{\phi}\alpha^2 \left\langle \gb^k, \frac{2}{\phi}\gb^{k-1} - \frac{\phi-1}{\phi}\gb^k \right\rangle = \frac{4}{\phi(\phi-1)}\alpha \langle \gb^k - \gb^{k-1}, \zb^k-\zb^{k+1} \rangle\\
  - \frac{3-\phi}{\phi-1} \Vert \zb^{k+1}-\zb^k \Vert^2 + \left( \frac{3-\phi}{\phi-1} \frac{\alpha^2}{\phi^2} - \frac{4}{\phi-1}\frac{\alpha^2}{\phi^2} \right) \Vert \gb^k - \gb^{k-1} \Vert^2.
\end{multline*}
We can deduce~\eqref{eq:crazy} by simplifying the above equation. The statement of the lemma is obtained by combining~\eqref{eq:crazy} and~\eqref{eq:firsty}.
\end{proof}

\begin{proof}[of Theorem~\ref{thm:graal-weak-minty}]
  We deduce from~\eqref{eq:iterates-as-grad}
  \begin{equation}
    \label{eq:grad}
    \frac{\phi-1}{\phi^2}\alpha^2 \Vert \gb_k \Vert^2 = \frac{\alpha^2}{\phi^2(\phi-1)} \Vert \gb^k - \gb^{k-1} \Vert^2 + \frac{2\alpha}{\phi(\phi-1)} \langle \gb^{k-1}-\gb^k, \zb^k-\zb^{k+1} \rangle + \frac{1}{\phi-1} \Vert \zb^k-\zb^{k+1} \Vert^2.
  \end{equation}
  Adding~\eqref{eq:grad} to~\eqref{eq:descent} gives
  \begin{multline}
    \label{eq:descent-minty-graal}
    \Vert \bar{\zb}^{k+1}-\zb^* \Vert^2 + \frac{\phi-1}{\phi}\alpha(\frac\alpha\phi-\rho) \Vert \gb^k \Vert^2 \le \Vert \bar{\zb}^k-\zb^* \Vert^2   +  \frac{2}{\phi (\phi-1)}\alpha \langle \gb^k - \gb^{k-1}, \zb^k-\zb^{k+1} \rangle\\
  - \frac{2-\phi}{\phi-1} \Vert \zb^{k+1}-\zb^k \Vert^2 - \frac{\alpha^2\phi}{\phi^2(\phi-1)} \Vert \gb^k - \gb^{k-1} \Vert^2
  \end{multline}
  By Young's inequality we deduce
  \begin{equation}
    \label{eq:young}
    \frac{2}{\phi (\phi-1)}\alpha \langle \gb^k - \gb^{k-1}, \zb^k-\zb^{k+1} \rangle \le \frac{\alpha}{\phi(\phi-1)L} \Vert \gb^k-\gb^{k-1} \Vert^2 + \frac{\alpha L}{\phi(\phi-1)} \Vert \zb^{k+1}-\zb^k \Vert^2.
  \end{equation}
  Combining~\eqref{eq:young} and~\eqref{eq:descent-minty-graal} yields
  \begin{multline*}
    \Vert \bar{\zb}^{k+1}-\zb^* \Vert^2 + \frac{\phi-1}{\phi}\alpha\left(\frac\alpha\phi-\rho\right) \Vert \gb^k \Vert^2 \le \Vert \bar{\zb}^k-\zb^* \Vert^2 \\
  - \left(\frac{2-\phi}{\phi-1} - \frac{\alpha L}{\phi(\phi-1)}\right) \Vert \zb^{k+1}-\zb^k \Vert^2 + \left(\frac{\alpha}{L \phi(\phi-1)}  - \frac{\alpha^2\phi}{\phi^2(\phi-1)}  \right) \Vert \gb^k - \gb^{k-1} \Vert^2.
  \end{multline*}
  Therefore, in order to telescope we require
  \begin{equation}
    \label{eq:condition-graal-minty}
    2\phi - \phi^2 - \alpha L \ge \alpha L - \alpha^2 L^2
  \end{equation}
  and $2\phi - \phi^2 \ge \alpha L$ for the last term to be nonnegative. The condition~\eqref{eq:condition-graal-minty} can be simplified to
  \begin{equation*}
    \frac{2-\phi}{L} \ge \alpha,
  \end{equation*}
  and the nonnegativity condition becomes redundant. We observe that, if $\rho < \frac{1}{L}$, then we can pick $\rho$ close enough to one such that $\frac{2-\phi}{\phi L} > \rho$.
  The final bound we obtain by
  \begin{equation*}
    \left(\frac{\alpha}{L \phi(\phi-1)}  - \frac{\alpha^2\phi}{\phi^2(\phi-1)}  \right) \Vert \gb^k - \gb^{k-1} \Vert^2 \le \left(\frac{\alpha L}{\phi(\phi-1)}  - \frac{\alpha^2L^2}{\phi(\phi-1)}  \right) \Vert \zb^k - \zb^{k-1} \Vert^2
  \end{equation*}
  where
  \begin{equation}
    \frac{\alpha L}{\phi(\phi-1)}  - \frac{\alpha^2L^2}{\phi(\phi-1)}  = \frac{\phi-2 - {(\phi-2)^2}}{\phi(\phi-1)} = \frac{3\phi -\phi^2-2}{\phi(\phi-1)} \overset{\phi\ge 1}{\le} \frac{2\phi -2}{\phi(\phi-1)} = \frac{2}{\phi}.
  \end{equation}
\end{proof}

\subsection{Adaptive step size}%
\label{sub:}

The analysis of~\eqref{eq:agraal} relies on a similar energy function as before. So with slight abuse of notation we (re-)define for the purpose of this section
\begin{equation*}
 \L(\zb^{k+1}, \zb) :=\frac{\phi}{\phi-1}\Vert \bar{\zb}^{k+1}-\zb \Vert^2 + \frac{\theta_k}{2} \Vert \zb^{k+1}-\zb^k \Vert^2.
\end{equation*}

\begin{proof}[of Theorem~\ref{thm:agraal-weak-minty}]
  The first few steps of our analysis follow the one presented in~\citep{malitsky2020golden} so we do not reproduce them here. We continue from (34) in~\citep{malitsky2020golden}, which reads
  \begin{equation*}
    \L(\zb^{k+1}, \zb)  \le \L(\zb^k, \zb) + \Big( \theta_k - 1 - \frac{1}{\phi} \Big)\alpha_k^2 \Vert F(\zb^k) \Vert^2 - \theta_k \Vert \zb^k - \bar{\zb}^k \Vert^2 - \alpha_k \langle F(\zb^k), \zb^k -\zb \rangle.
  \end{equation*}
  By using the weak Minty property~\eqref{eq:weak-minty} this reduces to
  \begin{equation}
    \label{eq:descent-adaptive}
    \L(\zb^{k+1}, \zb^*)   \le \L(\zb^k, \zb^*) + \alpha_k \left( \Big( \theta_k - 1 - \frac{1}{\phi} \Big)\alpha_k + \rho\right)\Vert F(\zb^k) \Vert^2 - \theta_k \Vert \zb^k - \bar{\zb}^k \Vert^2.
  \end{equation}
  Note that
  \begin{equation}
    \label{eq:extra-term-adaptive}
    \theta_k \Vert \zb^k - \bar{\zb}^k  \Vert^2 = \frac{\theta_k}{\phi^2}  \Vert  \zb^k - \bar{\zb}^{k-1}   \Vert^2 = \frac{\theta_k \alpha_{k-1}^2}{\phi^2} \Vert F(\zb^{k-1}) \Vert^2 = \frac{\alpha_{k-1}\alpha_k}{\phi} \Vert F(\zb^{k-1}) \Vert^2.
  \end{equation}
  Plugging~\eqref{eq:extra-term-adaptive} into~\eqref{eq:descent-adaptive} yields
  \begin{equation}
    \label{eq:descent-adaptive-simplified}
    \L(\zb^{k+1}, \zb^*) + \alpha_k\left(\frac{\alpha_{k-1}}{\phi}\Vert F(\zb^{k-1}) \Vert^2 + \left(\alpha_k\Big( 1 + \frac{1}{\phi} - \theta_k  \Big)- \rho \right)\Vert F(\zb^k) \Vert^2\right) \le \L(\zb^k, \zb^*).
  \end{equation}
  If for all $k$
  \begin{equation*}
    \alpha_k\left(\alpha_k\Big( 1 + \frac{1}{\phi} - \theta_k  \Big)- \rho \right) > 0,
  \end{equation*}
  which can be ensured, by the fact that $\theta_k\le \incr\phi$, if
  \begin{equation}
    \label{eq:agraal-weak-minty-condition-proof}
    \alpha_k\Big( 1 + \frac{1}{\phi} - \incr \phi  \Big)- \rho > 0,
  \end{equation}
  holds, then after telescoping, where we unroll the recursion until $k=1$ as argued in the proof of Lemma~\ref{lem: one_it_adap}, to obtain
  \begin{equation*}
    \sum_{i=1}^{k}\alpha_i\left((1 + \frac{1}{\phi} - \incr \phi)\alpha_i - \rho \right) \Vert F(\zb^i) \Vert^2  + \L(\zb^{k+1}, \zb^*)  \le \L(\zb^1, \zb^*).
  \end{equation*}
\end{proof}

\begin{remark}
 Note that in order to guarantee just $\liminf_k \Vert F(\zb^k) \Vert=0$ it is sufficient to ask for
 \begin{equation*}
   \sum_{i=1}^{k}\alpha_i\left((1 + \frac{1}{\phi} - \incr \phi)\alpha_i - \rho \right) \to \infty,
 \end{equation*}
 meaning that we can allow for arbitrarily many step size to not fulfill the condition $\delta>0$, but for the sequence on average.
\end{remark}

In the presence of bounded iterates we were able to relax the conditions of Theorem~\ref{thm:agraal-weak-minty}. Let $M$ denote the diameter of the ball containing the iterates.
\begin{proof}[of Corollary~\ref{cor:agraal-weak-minty-bounded-iter}]
  After telescoping~\eqref{eq:descent-adaptive-simplified} we get
  \begin{equation}
    \label{eq:use-extra}
    \begin{aligned}
      \L(\zb^{k+1}, \zb^*)+ \sum_{i=1}^{k} \alpha_i\left(\frac{\alpha_{i+1}}{\phi} + \alpha_i\Big(1 + \frac{1}{\phi} - \theta_i \Big)- \rho \right)\Vert F(\zb^i) \Vert^2  &\le  \L(\zb^1, \zb^*) + \frac{\alpha_{k+1}\alpha_k}{\phi} \Vert F(\zb^k) \Vert^2.
    \end{aligned}
  \end{equation}
  Let us first observe that the nonnegativity of the factor in front of $\Vert F(\zb^i) \Vert^2$ can be guaranteed via
  \begin{equation*}
    \frac{\alpha_{k+1}}{\phi} + \alpha_k \Big(1+ \frac{1}{\phi} - \theta_k \Big) \ge \frac{\alpha_{k+1}}{\phi} + \alpha_k \Big(1+ \frac{1}{\phi} - \incr \phi \Big),
  \end{equation*}
  since $\theta_k \le \incr \phi$.
  Also, the last term on the right hand side of~\eqref{eq:use-extra} can be bounded via
  \begin{equation*}
    \alpha_{k+1}\alpha_k \Vert F(\zb^k) \Vert^2 \le \frac{\alpha_{k+1}}{\alpha_k} \Vert \zb^{k+1} -\bar{\zb}^k \Vert^2 \le \incr \Vert \zb^{k+1} -\bar{\zb}^k \Vert^2.
  \end{equation*}
  Thus we obtain
  \begin{equation*}
      \sum_{i=1}^{k} \alpha_i\left(\frac{\alpha_{i+1}}{\phi} + \alpha_i\Big(1 + \frac{1}{\phi} - \incr\phi \Big)- \rho \right)\Vert F(\zb_i) \Vert^2  \le  \L(\zb^1, \zb^*) + \frac{\incr}{\phi} \Vert \zb^{k+1}- \bar{\zb}^k \Vert^2,
  \end{equation*}
  where the last term on the right remains bounded due to the assumed boundedness of the iterates.
  To deduce the precise constant we observe
  \begin{equation*}
    \frac{\phi}{\phi-1}\Vert \bar{\zb}^1-\zb \Vert^2 + \frac{\theta_0}{2} \Vert \zb^1-\zb^0 \Vert^2 + \frac{\incr}{\phi} \Vert \zb^{k+1}- \bar{\zb}^k \Vert^2 \le 2 M^2 + 2 M^2 + 2 M^2.
  \end{equation*}
  Thus,
  \begin{equation*}
    \min_{i=1,\dots,k} \Vert F(\zb_i) \Vert^2  \le  \frac{6}{\delta \sum_{i=1}^{k} \alpha_i} M^2.
  \end{equation*}
\end{proof}

\section{Details on experiments}%
\label{sec:gop3}

\subsection{Monotone problems}%

For the left plot in Fig.~\ref{fig:constant} we use the Lagrangian formulation of a linearly constrained quadratic program
\begin{equation*}
  L(x,y) = \frac12 x^T Hx - h^T x - \langle Ax-b, y \rangle,
\end{equation*}
where $x,y,h,b \in \R^d, A \in \R^{d\times d}$ with $d=100$.
We use the parametrization proposed in~\citep{lower-bounds-bilinear} and further studied in~\citep{acc-gradient-norm-minimax} which provides a particularly difficult instance.
The middle and right plot in Fig.~\ref{fig:constant} are special instances of bilinear matrix games of the form
\begin{equation*}
  \min_{x\in \Delta^d} \max_{y\in \Delta^d} \, x^T A y,
\end{equation*}
where $\Delta^d$ denotes the $d$-dimensional unit simplex $\{x \in \R^d : x\ge0,\, \sum_{i=1}^{d}x_i = 1\}$. For our experiments we used $d=50$.

\subsection{Weak minty experiments}%

\subsubsection{Algorithms}%

For~\eqref{eq:agraal} $\phi$ is usually given in the legend, except for Fig.~\ref{fig:spotlight} where we used the default $\phi=1.5$. If $\gamma$ is not given in the legend we use the theoretical upper bound $1/\phi + 1/\phi^2$.

For CurvatureEG$+$ we use a backtracking linesearch initialized with $\nu \Vert JF(\zb^k) \Vert^{-1}$, where we use $\nu=0.99$ and $JF$ denotes the Jacobian of $F$. We ignore the extra cost of this initialization but do count the extra gradient evaluations from the backtracking, where in every step the step size is decreased by $\tau=0.9$.

\subsubsection{Polar Game}%
\label{sub:polar-game}

The unique solution for this problem is in the origin. The Lipschitz constant $L$ and the weak Minty parameter $\rho$ we approximate numerically, via a grid search on the interval $[-1.1, 1.1] \times [-1.1, 1.1]$.
In Fig.~\ref{fig:polar-game} the value $a=1/3$ is used whereas in Fig.~\ref{fig:polar-game-adaptive} is given by $a=3.7$. For $a=1/3$ we get $L \approx 6.94$ and $\rho \approx 0.09$, whereas for $a=3$ we compute $L\approx 61.4$ and $\rho \approx 0.72$. In the latter case we observe from the result of the linesearch, see Fig.~\ref{fig:polar-game-adaptive}, that these global estimates are quite pessimistic but even locally the necessary condition $\rho < \alpha_k$ is not satisfied for CurvatureEG$+$, which is why we observe its divergence.

\vskip 0.2in
\bibliography{graal_refs}

\end{document}